\documentclass{amsart}
\usepackage{amscd,amssymb,amsopn,amsmath,amsthm,graphics,amsfonts,enumerate,verbatim,calc}
\usepackage[dvips]{graphicx}
\usepackage[colorlinks=true,linkcolor=red,citecolor=blue]{hyperref}
\usepackage[all]{xy}

\usepackage{tikz}
\usepackage{tikz}

\usepackage{tikz}

\usepackage{hyperref}

\usepackage{amssymb}
\usepackage{amstext}
\usepackage{amsmath}
\usepackage{amscd}
\usepackage{amsthm}
\usepackage{amsfonts}

\usepackage{graphicx}
\usepackage{latexsym}

\usepackage{array}
\newcolumntype{C}{>{$}c<{$}}

\usepackage{makecell}
\usepackage{adjustbox}

\calclayout

\newcommand{\rt}{\rightarrow}
\newcommand{\lrt}{\longrightarrow}

\newcommand{\st}{\stackrel}

\newcommand{\la}{\lambda}
\newcommand{\La}{\Lambda}
\newcommand{\Ga}{\Gamma}

\newcommand{\Z}{\mathbb{Z} }

\newcommand{\CA}{\mathcal{A} }
\newcommand{\CC}{\mathcal{C} }

\newcommand{\CE}{\mathcal{E}}

\newcommand{\CG}{\mathcal{G} }

\newcommand{\CK}{\mathcal{K} }

\newcommand{\CM}{\mathcal{M} }

\newcommand{\CQ}{\mathcal{Q} }

\newcommand{\CS}{\mathcal{S} }
\newcommand{\CT}{\mathcal{T} }
\newcommand{\CX}{\mathcal{X} }
\newcommand{\CY}{\mathcal{Y} }

\newcommand{\CU}{\mathcal{U}}

\newcommand{\CB}{\mathcal{B} }
\newcommand{\CH}{\mathcal{H}}

\newcommand{\mmod}{{\rm{{mod\mbox{-}}}}}

\newcommand{\prj}{{\mathsf{prj}\mbox{-}}}

\newcommand{\cm}{{\mathsf{CM}}}

\newcommand{\md}{{\mathsf{mod}}}

\newcommand{\Ker}{{\rm{Ker}}}

\newcommand{\rad}{{\rm{rad}}}

\newcommand{\uEnd}{\underline{\mathsf{End}}}
\newcommand{\gpr}{{\mathsf{Gprj}}}
\newcommand{\ugpr}{\underline{\mathsf{Gprj}}}

\newcommand{\mon}{\mathsf{mon}}
\newcommand{\ad}{\mathsf{add}}

\newcommand{\Hom}{{\rm{Hom}}}
\newcommand{\Ext}{{\rm{Ext}}}
\newcommand{\End}{\mathsf{{End}}}

\theoremstyle{plain}
\newtheorem{theorem}{Theorem}[section]
\newtheorem{corollary}[theorem]{Corollary}
\newtheorem{lemma}[theorem]{Lemma}

\newtheorem{proposition}[theorem]{Proposition}

\theoremstyle{definition}
\newtheorem{definition}[theorem]{Definition}
\newtheorem{example}[theorem]{Example}

\newtheorem{remark}[theorem]{Remark}
\newtheorem{setup}[theorem]{Setup}

\newtheorem{s}[theorem]{}

\theoremstyle{plain}

\theoremstyle{definition}

\numberwithin{equation}{section}

\begin{document}

\title[G-semisimple algebras]{ G-semisimple algebras}

\author[Rasool Hafezi and Abdolnaser Bahlekeh ]{Rasool Hafezi and Abdolnaser Bahlekeh}

\address{School of Mathematics and Statistics, Nanjing University of Information Science and Technology, Nanjing, 210044, China}

\email{hafezi@nuist.edu.cn}
\address{Department of Mathematics, Gonbade-Kavous University, Postal Code:4971799151, Gonbad-e-Kavous, Iran}
\email{bahlekeh@gonbad.ac.ir}

\thanks{}


\begin{abstract}
Let $\La$ be an Artin algebra and $\md\mbox{-} (\ugpr\mbox{-}\La)$ the category of finitely presented functors over the stable category $\ugpr\mbox{-}\La$ of finitely generated Gorenstein projective $\La$-modules. This paper deals with those algebras $\La$ in which $\md\mbox{-} (\ugpr\mbox{-}\La)$ is a semisimple abelian category, and we call G-semisimple algebras. We study some basic properties of such algebras. In particular, it will be observed that the class of G-semisimple algebras contains important classes of algebras, including gentle algebras and more generally quadratic monomial algebras. Next, we construct an epivalence from the stable category of Gorenstein projective representations $\ugpr(\CQ, \La)$ of a finite acyclic quiver $\CQ$ to the category of representations ${\rm rep}(\CQ, {{\ugpr}} \mbox{-} \La)$ over  ${{\ugpr}} \mbox{-} \La$, provided $\La$ is a G-semisimple algebra over an algebraic closed field. Using this, we will show that the path algebra $\La\CQ$ of the G-semisimple algebra $\La$ is $\cm$-finite if and only if $\CQ$ is Dynkin.
In the last part, we provide a complete classification of indecomposable Gorenstein projective representations within $\gpr(A_n, \La)$ of the linear quiver $A_n$ over a G-semisimple algebra $\La$. We also determine almost split sequences in $\gpr(A_n, \La)$ with certain ending terms. We apply these results to obtain insights into the cardinality of the components of the stable Auslander-Reiten quiver $\gpr(A_n, \La)$.

\end{abstract}

\maketitle
\tableofcontents

\section{Introduction }
Let $\La$ be an Artin algebra and $\text{H}(\La)$ the morphism category of $\La$. Assume that $\CX$ is a quasi-resolving subcategory of $\md\mbox{-}\Lambda$, the category of all finitely generated right $\La$-modules. Denote by $\CS_{\CX}(\La)$ the subcategory of $\text{H}(\La)$ consisting of those monomorphisms with terms in $\CX$ such that whose cokernels lie in $\CX$. In the case $\CX=\md\mbox{-}\Lambda$, the subcategory $\CS_{\md\mbox{-}\Lambda}(\La)$, simply denoted by $\CS(\La)$, becomes the submodule category which has been studied extensively by Ringel and Shmidmeier in \cite{RS1, RS2}. Ringel and Zhang \cite{RZ1} established a surprising connection between $\CS(k[x]/{(x^n)})$ and the module category $\Pi_{n-1}$ of type $A_{n-1}$ by defining two functors which originally come from the works of Auslander-Reiten and Li-Zhang. It is kown that $\Pi_{n-1}$ is the stable Auslander algebra of representation-finite algebra $k[x]/{(x^n)}$, see for example \cite{DR}. So the functors appeared in \cite{RZ1} are indeed functors from $\CS(k[x]/{(x^n)})$ to $\text{Aus}(\underline{\md}\mbox{-}k[x]/{(x^n)})$. Recall that the stable Auslander algebra $\text{Aus}(\underline{\md}\mbox{-}\La)$ of a representation-finite algebra $\La$ is the endomorphism algebra $\uEnd_{\La}(M)$, where $M$ is a representation generator for $\La$, that is, the additive closure ${\ad}\mbox{-}M$ is equal to $\md\mbox{-}\La$. These facts lead the first author to define a functor $\Psi_{\CX}$ from $\CS_{\CX}(\La)$ to $\md\mbox{-}\underline{\CX}$, for a quasi-resolving subcategory $\CX$ of $\md\mbox{-}\La$, see \cite[Construction 3.1]{H}. This functor is a relative version
of one of the two functors studied in \cite{RZ1}, and also \cite{E}, in a functorial language. According to \cite[Theorem 3.2]{H}, the functor $\Psi_{\CX}$ induces an equivalence between the quotient category $\CS_{\CX}(\La)/{\CU}$ and $\md$-$\underline{\CX}$, where $\CU$ is the ideal generated by the objects of the form $(X\st{1}\rt X)$ and $(0\rt X)$ for all $X\in\CX$. Here $\md$-$\underline{\CX}$ is the category of finitely presented functors on $\underline{\CX}$.
Now, if $\CX$ is representation-finite, i.e. $\CX=\ad$-$X$ for some $X\in\CX$, then $\md$-$\underline{\CX}\cong\md$-${\rm Aus}(\underline{\CX})$, where ${\rm Aus}(\underline{\CX})=\uEnd_{\La}(X)$. This yields that, the representation theoretic properties of $\CS_{\CX}(\La)$ via the functor $\Psi_{\CX}$ can be approached by the module category over the associated relative stable Auslander algebra.

Assume that $\gpr$-$\La$ is the category of finitely generated $\La$-modules, which is known to be a resolving subcategory of $\md$-$\La$. So, there is a functor $\Psi_{\gpr\mbox{-}\La}:\CS_{\gpr\mbox{-}\La}(\La)\lrt\md$-$(\ugpr$-$\La)$. In particular, if $\La$ is $\cm$-finite with $G$ an additive generator, then $\Psi_{\gpr\mbox{-}\La}$ induces an equivalence $\md\mbox{-}(\ugpr\mbox{-}\La)\cong\md\mbox{-}{\rm Aus}(\ugpr\mbox{-}\La)$. To simplify the notation, we denote $\CS_{\gpr\mbox{-}\La}(\La)$ by $\CS(\gpr\mbox{-}\La)$. It is known that ${\rm Aus}(\ugpr\mbox{-}\La)=\uEnd_{\La}(G)$, called the stable Cohen-Macaulay Auslander algebra of $\La$, is a self-injective algebra. This yields that the properties of the category of modules over the self-injective algebra can be transferred to the monomorphism category $\CS(\gpr\mbox{-}\La)$. This is of particular interest because the category of modules over a self-injective algebra is well-understood, see for example \cite{DI, SY}. Motivated by these observations, we study the monomorphism category $\CS(\gpr\mbox{-}\La)$ via the functor category $\md$-$(\ugpr$-$\La)$, where $\La$ is $\cm$-finite, and so, it is equivalent to the module category of a self-injective algebra. The initial step in this direction, at least in the homological dimensions sense, is when the global projective dimension of $\md\mbox{-}(\ugpr\mbox{-}\La)$ is zero, i.e., $\md\mbox{-}(\ugpr\mbox{-}\La)$ is a semisimple abelian category. We call an algebra with this property a G-{\it semisimple algebra.}  The notion of G-semisimple algebras was first introduced in an unpublished work \cite{H1} by the first author. In this current paper, we build upon and enhance the results established in \cite{H1} by improving and extending them.

The next step, will be that if the stable Cohen-Macaulay Auslander algebra is a self-injective Nakayama algebra, then does it reflect some expected properties of $\CS(\gpr\mbox{-}\La)$? This will be done in a separate work.

In Section 3, we give some basic properties of these algebras. It is shown that any G-semisimple algebra is $\cm$-finite, see Proposition \ref{CM-finite}. It will be also seen that over a G-semisimple algebra $\La$, the indecomposable non-projective objects of $\CS(\gpr$-$\La)$ can be computed within those in $\gpr$-$\La$. This, in particular yields that $\CS(\gpr$-$\La)$ is representation finite, see Proposition \ref{pro11}. We shall also describe the singularity category over Gorenstein G-semisimple algebras, see Proposition \ref{proposition 2.8}. A classification of indecomposable non-projective Gorenstein projective modules over G-semisimple algebras will be given which generalizes the ones in \cite{CSZ} for quadratic monomial algebras and \cite{Ka} for gentle algebras to G-semisimple algebras (see Remark \ref{remark 2.7} for details). The classification is given by defining an equivalence relation on the indecomposable non-projective modules in $\gpr$-$\La$. Moreover,
a characterization of 1-Gorenstein G-semisimple algebras in terms of the radical of their projective modules will be given, see Proposition \ref{RadGorp}. Also, some examples of such algebras are presented, see Example \ref{123}. As indicated above, a G-semisimple algebra is $\cm$-finite, so such an algebra $\La$ is nothing else to say that the associated stable Cohen-Macaulay algebra of $\La$ is a semisimple algebra. Such an observation was investigated in the remarkable paper of Auslander and Reiten, where they showed that for a dualizing $R$-variety $\mathcal{C}$: the global dimension of $ \md\mbox{-}(\underline{\md}\mbox{-}\CC)$ is zero if and only if $\md\mbox{-}\CC$ is Nakayama and whose lowey length is at most two \cite[Theorem 10.7]{AR1}. The common case with our study is when $\mathcal{C}=\mathsf{prj}\mbox{-}\La$ and $\La$ is a self-injective algebra. It will be observed that the class of G-semisimple algebras contains important classes of algebras in the representation theory such as the class of gentle algebras, and more generally the class of quadratic monomial algebras, see Remark \ref{Remark 2.8}. The study of gentle algebras was initiated by Assem and Skowro\'{n}ski \cite{ASk} in order to study iterated tilted algebras of type $\mathbb{A}$. Viewing gentle algebras as G-semisimple algebras provides them a functorial approach. Even though the provided functorial approach is not completely determined by gentle algebras, however it might be helpful to consider them in the larger class of G-semisimple algebras.

During the last decades, there has been a rising interest in the monomorphism category of a ring. The study of the monomorphism category, which is known also a submodule category, goes back to G. Birkhoff \cite{Bi}, in which he has initiated to classify the indecomposable objects of the submodule category of $\Z/{(p^{n})}$, with $n\geq 2$ and $p$ a prime number. The study of submodule categories was revitalized by the systematic and deep work of Ringel and Schmidmeier \cite{RS1, RS2}. Thereafter, links to other areas of algebra such as singularity category and weighted projective lines \cite{KLM, S} and Gorenstein homological algebra \cite{Z, LZh2, ZX} were also detected. Recently, in order to describe the indecomposable objects of the (separated) monomorphism category, Gao et.al. \cite{GKKP} constructed an epivalence, i.e. a full and dense functor that reflects isomorphisms. Precisely, let $\CQ$ be  a finite acyclic quiver and $\CB$ an abelian category with enough injectives. Assume that $\mathsf{mono}(\CQ, \CB)$ is the (separated) monomorphism category of $\CQ$ over $\CB$. Then they proved that the canonical functor from the injectively stable category $\overline{\mathsf{mono}}(\CQ, \CB)$ to the category of representations of $\CQ$ over $\bar{\CB}$, ${\rm rep}(\CQ, \bar{\CB})$ is an epivalence. Moreover, if $\CQ$ has at least one arrow, then it is equivalence if and only if $\CB$ is hereditary, see \cite[Theorem A]{GKKP}. Inspired by this, in Section 4, we show that  there is an epivalence from the stable category of Gorenstein projective representations $\ugpr(\CQ, \La)$ of a finite acyclic quiver $\CQ$ to the category of representations ${\rm rep}(\CQ, {{\ugpr}} \mbox{-} \La)$ over ${{\ugpr}} \mbox{-} \La$, provided $\La$ is a G-semisimple algebra over an algebraic closed field, see Theorem \ref{main}. As an interesting application, we show that if $\La$ is a G-semisimple algebra over an algebraically closed field, then the path algebra $\La\CQ$ is $\cm$-finite if and only if $\CQ$ is Dynkin. Also, the number of indecomposable non-projective objects of $\ugpr(\CQ, \La)$, with $\CQ$ Dynkin and $\La$ G-semisimple, is determined in terms of positive roots, see Corollary \ref{dynkin}.

The paper's final section, Section 5, deals with the stable Auslander-Reiten quiver of $\CS_n(\gpr$-$\La)$ with $\La$ a G-semisimple algebra. Here $\CS_n(\gpr$-$\La)$ is the category of Gorenstein projective representations $\gpr(A_n, \La)$, and so, $\CS_2(\gpr$-$\La)=\CS(\gpr\mbox{-}\La)$.
In this direction, not only the existence of almost split sequences in $\CS_n(\gpr$-$\La)$ is proved but also the structure of such sequences with certain ending terms is recognized. Also, the isomorphism classes of indecomposable objects of this category, are completely described. The main result of this section gives some information concerning the cardinality of the stable Auslander-Reiten quiver of $\CS_n(\gpr$-$\La$), see Theorem \ref{Theorem 4.10}. Precisely, we show that the cardinality is divisible by $n+1$ (resp. $\frac{n+1} {2}$), whenever $n$ is even (resp. odd).

{\bf Notation and Convention.}
Throughout the paper, $\La$ is an Artin algebra, and by a module, we always mean a finitely generated right $\La$-module. The category of all (finitely generated) $\La$-modules will be denoted by $\md\mbox{-}\Lambda$. For a module $M$, consider a short exact sequence $0 \rt \Omega_{\La}(M)\st{i}\rt P \st{p}\rt M\rt 0$ with $P$ a
projective cover of $M$. The module $\Omega_{\La}(M)$ is then called a syzygy module of $M$. An $n$-th syzygy of $M$ will be
denoted by $\Omega^n_{\La}(M)$, for $n\geqslant 2$. For a module $G$ in $\gpr$-$\La$, $\Omega^{-1}_{\La}(G)$ denotes the cokernel of a minimal left $\mathsf{prj}\mbox{-}\La$-approximation of $G$. Similarly, one can define $\Omega^{-n}_{\La}(G)$ for any $n\geqslant 2.$ It should be noted that, if $\La$ is assumed to be a selfinjective algebra, then $\Omega^{n}_{\La}G$, is the same as the $n$-th cosyzygy of $G$. Moreover, if $G$ is indecomposable, then it is easily seen that $\Omega^{-1}_{\La}\Omega_{\La}M\simeq M\simeq\Omega_{\La}\Omega^{-1}_{\La}M$. For a sufficiently nice category $\CC$, we denote by $\tau_{\CC}$ the Auslander-Reiten translation in $\CC$. By $\text{ind}\mbox{-}\CC$
we mean the subcategory of all indecomposable objects in $\CC.$ For simplicity, for the cases $\CC=\md\mbox{-}\La$ and $\gpr$-$\La$, we write $\tau_{\La}$ and $\tau_{\CG}$ instead, respectively. We refer to \cite{AuslanreitenSmalo, AS, Li, Kr} for a background on the Auslander-Reiten theory,
in particular the last three references concerning the Auslander-Reiten theory in subcategories which we are mostly dealing with in this paper. For the sake of simplicity, we often identify a module with its isomorphism class.

\section{preliminaries}

In this section, we recall the definitions of several basic notions that we deal within this paper.

\begin{s}{\sc Almost split sequences.} Let $(\CC, \CE)$ be an exact category. A morphism
$f: B\rt C$ in $\CC$ is said to be right almost split provided it is not a split epimorphism and
every morphism $h: X\rt C$ which is not a split epimorphism factors through $f$. {Dually, a morphism
$g: A\rt B$ is left almost split if it is not a split monomorphism and each morphism $h: A\rt C$ which
is not a split monomorphism factors through $g$. A short exact sequence $\eta:0\rt A\st{f}\rt B\st{g}\rt C\rt 0$
in $\CE$ is called almost split, provided that $f$ is left almost split and ${g}$ is right almost split, see \cite{Li} and also \cite{INP} in the context of extriangulated categories.
We remark that, since this sequence is unique up to isomorphism for $A$ and for $C$,
we may write $A=\tau_{\mathcal{A}}{C}$ and $C=\tau_{\mathcal{A}}^{-1}A$. We call $A$ the Auslander-Reiten translation of $C$ in $\CC$.
It is known that if $\eta:0\rt A\st{f}\rt B\st{g}\rt C\rt 0$ is an almost split sequence, then $A$ and $C$ have local endomorphism ring, see \cite[Proposition II.4.4]{Au3}.}
We say that $\CC$ has right almost split sequences, if for any non-projective object $Y$ with local endomorphism ring,
there is an almost split sequence $\delta: 0\rt X\rt E\rt Y\rt 0$ in $\CE$. The notion of having left almost split sequences
is defined dually. Now $\CC$ is said to have almost split sequences, if it has left and right almost split sequences.
\end{s}

\begin{s}{\sc Gorenstein projective objects.} Assume that $\mathcal{A}$ is an abelian category with enough projectives. A complex $$\mathbf{P}^{\bullet}:\cdots\lrt P^{i-1}\st{d^{i-1}}\lrt P^i\st{d^i}\lrt P^{i+1}\lrt\cdots$$ of projective objects in $\mathcal{A}$ is said to be {\em totally acyclic}, provided it is acyclic and the Hom complex $\Hom_{\mathcal{A}}(\mathbf{P}^{\bullet}, Q)$ is also exact, for all projective objects $Q$ in $\mathcal{A}$. An object $M$ in $\mathcal{A}$ is called {\em Gorenstein projective} if it is a syzygy of a totally acyclic complex. The full subcategory of all Gorenstein projective objects in $\mathcal{A}$ will be denoted by $\gpr\mbox{-} \mathcal{A}$. Moreover, for the simplicity, if $\mathcal{A}=\md\mbox{-} \La$, then we denote it by $\gpr\mbox{-} \La$.

The concept of Gorenstein projective modules was first defined by
Maurice Auslander in the mid-sixties \cite{AB}, for finitely generated modules over a two-sided
Noetherian ring, called modules of $G$-dimension zero. Later they found important applications in
commutative algebra, algebraic geometry, singularity theory, and relative homological algebra.
Enochs and Jenda \cite{EJ} generalized the notion of Gorenstein projective modules to (not necessarily
finitely generated) modules over any ring $R$, under the name of Gorenstein projective modules. Studying (finitely generated) Gorenstein projective modules has
attracted attention in the setting of Artin algebras, see for instance \cite {AR4, AR5, CSZ, RZ1, RZ2} and references therein.

\end{s}

\begin{s}{\sc Functor categories.} Let $\CX$ be a skeletally small additive category. A (right) $\CX$-module is a contravariant additive functor from $\CX$ to the category of abelian groups.
The $\CX$-modules and natural transformations between them form an abelian category which is denoted by $\mathsf{Mod}$-$\CX$.
Recall from \cite{Au4} that an $\CX$-module $F$ is called {\em finitely presented} if there exists an exact sequence
$$\Hom_{\CX}(-, B)\lrt\Hom_{\CX}(-, B')\lrt F\lrt 0$$ with $B, B'\in\CX$. We denote by $\md\mbox{-}\CX$ the category of all finitely presented $\CX$-modules.
It is known that $\md$-$\CX$ is an abelian category if and only if $\CX$ admits weak kernels, see \cite[Chapter III, Section 2]{Au2}.
This happens, for example, when $\CX$ is a contravariantly finite subcategory of an abelian category $\CA$.
In the remainder of the paper, we shall sometimes write $(-, C)$ to indicate the representable functor $\CX(-, C)=\Hom_{\CX}(-, C)$.

Assume that $\CX$ is a contravariantly finite subcategory of $\md$-$\La$ of finite representation type that contains projectives. Assume that $X$ is an object of $\CX$ such that $\CX=\mathsf{add}\mbox{-}(X)$, where $\mathsf{add}\mbox{-}(X)$ the set of all direct summands of finite direct sums of copies of $X$. Then the relative Auslander algebra of $\La$, denoted ${\rm Aus}(\CX)$, is defined to be the endomorphism algebra $\End_{\La}(X)$, see \cite[Definition 6.3]{B}. It is known that ${\rm Aus}(\CX)$ is independent of the choice of the representation generator of $\CX$, up to Morita equivalence. In addition, there exists an equivalence of abelian categories $$e_{X}: \md\mbox{-}\CX\lrt\md\mbox{-}{\rm Aus}(\CX)$$ which is defined for any object $F\in\md\mbox{-}\CX$ by the evaluation functor of $F$ at $X$. This yields an equivalence of categories $$\underline{e}_{X}: \underline{\md}\mbox{-}\CX\lrt\md\mbox{-}{\rm Aus}(\underline{\CX}).$$Assume that $\La$ is of finite $\cm$-type, that is, $\gpr\mbox{-}\La$ is representation-finite. In this case, ${\rm Aus}(\ugpr\mbox{-}\La)$ is called the stable Cohen-Macaulay Auslander algebra of $\La$. Due to the latter equivalence, we can identify the category of modules over the stable Cohen-Macaulay Auslander algebra of $\La$ with the category of those finitely presented functors in $\md\mbox{-}(\gpr\mbox{-}\La)$ which vanish over the projective objects. Throughout the paper, we use freely this identification.
\end{s}

\begin{s}{\sc Morphism categories.} The morphism category of $\La$-modules denoted by ${\rm H}(\La)$ is a category
whose objects are all $\La$-homomorphisms $(A\st{f}\rt B)$ and morphisms are given by commutative squares. Moreover, the composition of morphisms is defined component-wise.
The category ${\rm H}(\La)$ is an abelian category. Indeed, it is equivalent to the category of finitely generated modules over the upper triangular matrix ring $T_2(\La)= \tiny {\left[\begin{array}{ll} \La & \La \\ 0 & \La \end{array} \right]}$.
The monomorphism category of $\La$, denoted by $\CS(\La)$, is the full subcategory of ${\rm H}(\La)$ consisting of all monomorphisms.
Since $\CS(\La)$ is an extension-closed subcategory of ${\rm H}(\La)$, it will be an exact category in the sense of Quillen, see \cite[Appendix A]{K1}.
The conflations in $\CS(\La)$ are those short exact sequences in ${\rm H}(\La)$ with terms in $\CS(\La)$. The full subcategory of $\CS(\La)$ whose objects are
all monomorphisms $(A\st{f}\rt B)$ for which $A, B, \mathsf{cok}f$ are Gorenstein projective will be denoted by $\CS(\gpr\mbox{-}\La)$.
\end{s}

\section{G-semisimple algebras} \label{Omega-algebras}
In this section, we will introduce the notion of G-semisimple algebras. Some basic properties and examples of G-semisimple algebras will be provided. In particular, indecomposable Gorenstein projective modules over a G-semisimple algebra are characterized. Next, this characterization is applied to describe the singularity categories over Gorenstein G-semisimple algebras.

Throughout this section, we assume that $\gpr\mbox{-} \La$ is contravariantly finite in $\md\mbox{-}\La.$ Recall that a subcategory $\CX$ of $\md\mbox{-}\La$ is resolving if it contains all projectives, closed under extensions, and the kernels of epimorphisms. On the other hand, since $\gpr\mbox{-} \La$ is resolving it also becomes a functorially finite subcategory of $\md\mbox{-}\La$ (due to \cite{KS}). So, by \cite{AS}, we can talk about the Auslander-Reiten theory in $\gpr\mbox{-} \La.$

By definition of the Gorenstein projective modules, one may observe that for each $G \in \gpr\mbox{-} \La$, $G^*=\Hom_{\La}(G, \La)$ is a
Gorenstein projective module in $\md\mbox{-}\La^{\rm{op}}.$ Then the duality $(-)^*:\prj \La \rt \prj \La^{\rm{op}}$ can be naturally
generalized to the duality $(-)^*:\gpr\mbox{-} \La \rt \gpr\mbox{-} \La^{\rm{op}}.$

Motivated by Section 5 of \cite{H1}, an unpublished paper by the first author, we state the following definition.

\begin{definition}
Let $\La$ be an Artin algebra. Then $\La$ is called a {\it G-semisimple algebra}, if $\md\mbox{-}(\ugpr\mbox{-} \La)$ is a semisimple abelian category, i.e. all whose objects are projective.
\end{definition}

In the sequel, we give some basic properties of G-semisimple algebras. First, we quote a couple of lemmas.

\begin{lemma}\label{Radical}
Let $P$ be an indecomposable projective object in $\mmod \La.$
\begin{itemize}
\item[$(i)$] If $f: G \rt \mathsf{rad}P$ is a minimal right $\gpr\mbox{-} \La$-approximation, then $i \circ f$ is a minimal right almost split morphism in $\gpr\mbox{-} \La$. Here, $i:\mathsf{rad}P\hookrightarrow P$ denotes the canonical inclusion.
\item[$(ii)$] If $g: G' \rt \mathsf{rad}P^*$ is a minimal right $\gpr\mbox{-} \La^{\rm{op}}$-approximation, then $(j \circ g)^*: P \rt (G')^*$ is a minimal left almost split morphism in $\gpr\mbox{-} \La$. Here, $j:\mathsf{rad}P^*\hookrightarrow P^*$ denotes the canonical inclusion in $\md\mbox{-}\La^{\rm op}$, and by abuse of notation, we identify $P^{**}$ and $P$ because of the natural isomorphism $P^{**}\simeq P$.
\end{itemize}
\end{lemma}
\begin{proof}
$(i)$ We only need to prove that for any non-split epimorphism $h:G_0 \rt P$, there exists a morphism $l:G_0 \rt G$	so that $i\circ f\circ l=h.$ Since $h$ is a non-split epimorphism, $\text{Im}( h)$ is a proper submodule in $P$, i.e. $\text{Im}(h) \subseteq \mathsf{rad}P.$ So we can write $h=i \circ s$, for some $s: G_0 \rt \mathsf{rad}P$. Now since $f$ is a right $\gpr$-$\La$-approximation, there is a morphism $l: G_0 \rt G$ such that $f \circ l=s$ and clearly it works for what we need. Moreover, it is easy to observe that $i \circ f$ is minimal by considering the fact that both $i$ and $f$ are minimal, with $i$ being an inclusion.

$(ii)$ The result follows from $(i)$ and the duality $(-)^*:\gpr$-$\La \rt \gpr$-$\La^{\rm{op}}.$
\end{proof}

\begin{lemma}\label{Lemma1}
Let $\La$ be a {\rm G}-semisimple algebra. Then any short exact sequence in $\gpr$-$\La$ is a direct sum of the short exact sequences of the form $0 \rt \Omega_{\La}(A) \rt P \rt A \rt 0$, $0 \rt 0 \rt B \st{1} \rt B \rt 0$ and $0 \rt C \st{1} \rt C \rt 0 \rt 0$ for some $A, B $ and $C$ in $\gpr$-$\La$.
\end{lemma}
\begin{proof}
Take a short exact sequence $0 \rt G_2 \rt G_1 \rt G_0 \rt 0$ in $\gpr$-$\La$. This sequence induces the exact sequence $$ 0 \rt (-, G_2) \rt (-, G_1) \rt (-, G_0) \rt F \rt 0 \ \ \ \ (1)$$
in $\md$-$(\gpr$-$\La).$ Since $\md$-$(\underline{\gpr}$-$\La)$ is a semisimple abelian category, $F \simeq (-, \underline{G})$ for some object $G$ in $\gpr$-$\La$. On the other hand, it is known that the minimal projective resolution of $(-, \underline{G})$ is as the following $$ 0 \rt (-, \Omega_{\La}(G)) \rt (-, P) \rt (-, G) \rt (-, \underline{G}) \rt 0 \ \ \ (2) $$ where $P \rt G $ is the projective cover of $G.$ By comparing sequences $(1)$ and $(2)$ as two projective resolutions of $F$ in $\md$-$(\gpr$-$\La)$, and using this fact that the latter is minimal, we get our result. Indeed, the fact we have used here is known in the homological algebra, that is, any projective resolution of an object in a perfect category is a direct sum of the minimal projective resolution of the given object and possibly some split exact complexes.
\end{proof}

Recall that an Artin algebra $\La$ is said to be $\cm$-finite, if there are only finitely many isomorphism classes of indecomposable finitely generated Gorenstein projective $\La$-modules.

\begin{proposition}\label{CM-finite}
Let $\La$ be a {\rm G}-semisimple algebra. Then $\La$ is $\cm$-finite.
\end{proposition}
\begin{proof}
Assume that $G$ is an indecomposable non-projective Gorenstein projective $\La$-module. According to Lemma \ref{Lemma1}, the almost split sequence in $\gpr$-$\La$ ending at $G$, is of the form $0 \rt \Omega_{\La}(G) \rt P \st{f} \rt G \rt 0$ in $\gpr$-$\La$, where $P\rt G$ is the projective cover of $G$. Suppose that $Q$ is an indecomposable direct summand of $P$. The natural injection $f_{\mid Q}: Q \rt G$ is irreducible by using the general facts in the theory of almost split sequences for subcategories, e.g. \cite[Theorem 2.2.2]{Kr}. In view of Lemma \ref{Radical} there is a minimal left almost split morphism $g: Q \rt Y$ in $\gpr$-$\La$. Hence $G$ is a direct summand of $Y.$ This indeed means that any indecomposable non-projective Gorenstein projective module is related to an indecomposable projective module via a minimal left almost split morphism. Now the uniqueness of the minimal left almost split morphisms, together with the fact that $\md\mbox{-}\La$ contains only finitely many indecomposable projective modules up to isomorphism, would imply the desired result.
\end{proof}

\begin{definition}
An indecomposable non-projective Gorenstein projective $\La$-module $G$ is called {\em G-semisimple}, provided that the canonical short exact sequence $0 \rt \Omega_{\La}(G)\rt P_G\rt G\rt 0$ is an almost split sequence in $\gpr$-$\La$.
\end{definition}
One may have the following easy characterization of G-semisimple algebras in terms of $G$-semisimple modules.

\begin{lemma}
$\La$ is a  G-semisimple algebra if and only if every Gorenstein projective module is  G-semisimple and $\La$ is $\cm$-finite.
\end{lemma}
\begin{proof}
First, assume that $\Lambda$ is a G-semisimple algebra. By Proposition \ref{CM-finite}, $\Lambda$ is $\cm$-finite. Moreover,
since because any almost split sequence does not split, Lemma \ref{Lemma1} yields that any Gorenstein projective module is G-semisimple.\\
Conversely, if $\Lambda$ is $\cm$-finite, as mentioned in the introduction, $\md$-$(\underline{\gpr}$-$\La$),
can be considered as a module category over some algebra.
Hence, there exists a finite filtration of subfunctors such that each factor is a simple functor. Every simple functor $S$ in $\md$-$(\underline{\gpr}$-$\La$) corresponds to an almost split sequence in $\gpr$-$\La$. More precisely, there exists an indecomposable module $M$ with an almost split sequence $0 \rightarrow \tau_{\mathcal{G}} M \rightarrow B \rightarrow M\rt 0$ such that there exists an exact sequence
$0 \rightarrow (-, \tau_{\mathcal{G}} M) \rightarrow (-, B) \rightarrow (-, M) \rightarrow S \rightarrow 0$
in $\md$-$(\gpr$-$\La)$. Since by our assumption $B$ is projective
it follows that $S = (-, \underline{M})$, which is a representable functor in $\md$-$(\ugpr$-$\La)$. Therefore, every simple module is projective in $\md$-$(\ugpr$-$\La)$. In conclusion, these results establish the desired result.

\end{proof}

The following result characterizes the indecomposable non-projective objects in $\CS(\gpr$-$\La)$.
\begin{proposition}\label{pro11}
Let $\La$ be a {\rm G}-semisimple algebra. Then, any indecomposable non-projective object in $\CS(\gpr\mbox{-}\La)$ is isomorphic to either $(0\rt G)$ or $(G\st{1}\rt G)$ or $(\Omega_{\La} G\st{i}\rt P)$, where $i$ the kerenl of the projective cover $P\rt G$, for some indecomposable non-projective Gorenstein projective module $G$. In particular, $\CS(\gpr$-$\La)$ is of finite representation type, or equivalently $T_2(\La)$ is $\cm$-finite.
\end{proposition}
\begin{proof}
Take an arbitrary indecomposable object $(X\st{f}\rt Y)$ in $\CS(\gpr$-$\La)$. It induces the short exact sequence $\lambda: 0 \rt X\st{f}\rt Y\rt \mathsf{cok}f\rt 0$ with all terms in $\gpr$-$\La.$ Lemma \ref{Lemma1} implies that the sequence $\lambda$ is isomorphic to
the short exact sequences of the form either $0 \rt \Omega_{\La}(G) \rt P \rt G \rt 0$ or $0 \rt 0 \rt G \st{1} \rt G \rt 0$ or $0 \rt G \st{1} \rt G \rt 0 \rt 0$, for some indecomposable module $G$ in $\gpr$-$\La.$ Hence we get the first assertion.\\ The second assertion follows from the first one, utilizing two key facts. Firstly, G-semisimple algebras are $\cm$-finite (as stated in Proposition \ref{CM-finite}). Secondly, the operation of taking a projective cover provides us with only a finite number of isomorphism classes of indecomposable objects in $\CS(\gpr$-$\La)$. Therefore, based on these observations, we can conclude that the proof is complete.
\end{proof}

\color{black}

Recall from \cite[Lemma 3.4]{C} that for a semisimple abelian category $\CA $ and an auto-equivalence $\Sigma $ on $\CA$, there is a unique triangulated structure on $\CA$ with $\Sigma$ being the translation functor. Indeed, all the triangles split. We denote the resulting triangulated category by $(\CA, \ \Sigma)$. We call a triangulated category semisimple, provided that it is  triangle equivalent to $(\CA, \ \Sigma)$ for some semisimple abelian category $\CA$ and an equivalence on it

Recall from \cite{RV} that a Serre functor for an $R$-linear triangulated category ($R$ is a commutative Artinian ring) $\CT$ is an auto-equivalence $S:\CT\rt \CT$
together with an isomorphism $D\Hom_{\CT}(X, -)\simeq \Hom_{\CT}(-, S(X))$ for each $X \in \CT$. Here $D$ stands for the ordinary duality.

\begin{proposition}\label{Omega-algebra}
Let $\La$ be a {\rm G}-semisimple algebra. Then the following assertions hold.
\begin{itemize}

\item[$(i)$] $\ugpr$-$\La$ is a semisimple triangulated category.

\item[$(ii)$] For any non-projective Gorenstein projective indecomposable module $G$, the (relative) Auslander-Reiten translation of $G$ in $\gpr$-$\La$ is the first syzygy $\Omega_{\La}(G).$

\item[$(iii)$] The Serre functor $S_{\CG}$ over $\ugpr$-$\La$ acts on objects by the identity, i.e. $S_{\CG}(G)\simeq G$ for each $G$ in $\ugpr$-$\La$.
\end{itemize}
\end{proposition}
\begin{proof}
$(i)$ Since $\La$ is G-semisimple, by Lemma \ref{Lemma1} any short exact sequence with all terms in $\rm{Gprj} \mbox{-} \La$ can be written as a direct sum of the short exact sequences of the form $0 \rt \Omega_{\La}(A) \rt P \rt A \rt 0$, $0 \rt 0 \rt B \st{1} \rt B \rt 0$ and $0 \rt C \st{1} \rt C \rt 0 \rt 0$ for some $A, B $ and $C$ in $\gpr$-$\La.$ Since triangles in $\ugpr$-$\La$, are induced by the short exact sequences in $\gpr$-$\La$, all possible triangles in $\ugpr$-$\La$ are a direct sums of the trivial triangles $K \rt 0 \rt K[1] \st{1} \rt K[1] $, $K \st{1} \rt K \rt 0 \rt K[1]$ and $0 \rt K \st{1} \rt K \rt 0.$
By the axioms of triangulated categories, every morphism $f:X \rt Y$ in $\ugpr$-$\La$ is completed to a triangle, hence, by the above, we can write as $f=\bigl(\begin{smallmatrix}
f_1 & 0 \\
0 & 0
\end{smallmatrix}\bigr):X_1\oplus Y_1 \rt X_2 \oplus Y_2$, where $f_1:X_1 \rt X_2$ is an isomorphism in $\ugpr$-$\La$. Define the natural injection $i: Y_1 \rt X_1 \oplus Y_1$ and the natural projection $X_2 \oplus Y_2 \rt Y_2$ as the kernel and the cokernel of $f$, respectively. Then this gives a semisimple abelian structure over $\ugpr$-$\La$. Now if we consider the equivalence $\Omega_{\La}:\ugpr$-$\La\rt \ugpr$-$\La$, then the resulting triangulated category $(\ugpr$-$\La, \Omega)$ is triangle equivalent to $\ugpr$-$\La$, so we get $(i).$

$(ii)$ Since any almost split sequence does not split, this assertion follows
from Lemma \ref{Lemma1}.

$(iii)$ This comes from $(ii)$ and using this point that $\tau_{\CG}=S_{\CG}\circ \Omega_{\La} .$
\end{proof}

The result below is obtained easily by using the duality $(-)^*:\gpr$-$\La \rt \gpr$-$\La^{\rm{op}}.$ So we skip its proof.

\begin{proposition}\label{Symmetry}
Let $\La$ be a  G-semisimple algebra. Then $\La^{\rm{op}}$ is also a  G-semisimple algebra.
\end{proposition}

Assume that $\La$ is a G-semisimple algebra. In view of Proposition \ref{CM-finite}, $\gpr$-$\La$ is of finite representation type. On the other hand, by \cite[Lemma 2.2]{C1}, the syzygy functor preserves the indecomposable Gorenstein projective modules. These facts yield that the non-projective Gorenstein projective modules are $\Omega_{\La}$-periodic, meaning that there exists an integer $n > 0$, depending on a given non-projective Gorenstein projective module $G$, such that $\Omega^n_{\La}(G)\simeq G$.

On indecomposable non-projective modules in $\gpr\mbox{-}\La$, we define an equivalence relation, as follows: for given two indecomposable objects $G$ and $G'$ in $\gpr\mbox{-}\La$, we write $G\sim G'$ if and only if there is some $n \in \mathbb{Z} $ such that $G\simeq \Omega^n_{\La}(G')$.

Let $\mathcal{C}(\Lambda)$ denote the set of all equivalence classes with respect to the relation $\sim$. Furthermore, for an indecomposable non-projective Gorenstein projective module $G$, we use $l(G)$ to represent the minimum number $n>0$ for which $\Omega^n_{\Lambda}(G)\simeq G$. It can be observed that the equivalence class $[G]$ consists of $l(G)$ indecomposable modules up to isomorphism.

Assume that $G$ is an indecomposable non-projective module in $\gpr\mbox{-}\La$. By considering Proposition \ref{Omega-algebra}$(ii)$, we can conclude that the $\tau_{\mathcal{G}}$-orbit of $G$ is equal to the equivalence class $[G]$ in $\mathcal{C}(\Lambda)$.

\begin{remark}\label{remark 2.7}
Let $\La=k\CQ/I$ be a monomial algebra, that is, the quotient algebra of the path algebra $k\CQ$	of a finite quiver $\CQ$ modulo the ideal $I$ generated by the paths of length at least 2. Following \cite[Definition 3.7]{CSZ}, we recall that a non-zero path $p$ in $\La$ is a perfect path, provided that there exists a sequence
$p = p_1, p_2,\cdots ,p_n, p_{n+1}= p$ of non-zero paths such that $(p_i, p_{i+1})$ are perfect pairs (see \cite[Definition 3.3]{CSZ}) for all $1\leqslant i \leqslant n$. If the given
non-zero paths $p_i$ are pairwise distinct, we refer to the sequence $p = p_1, p_2,...,p_n$ as a relation cycle for $p$, which has length $n$. It is proved in \cite[Theorem 4.1]{CSZ} that there is a bijection between the set of perfect paths
in $\La$ and the set of isoclasses of indecomposable non-projective Gorenstein projective $\La$-modules, which is a generalization of the classification over gentle algebras given in \cite{Ka}. The bijection is given by sending a perfect path $p$ to the $\La$-module $p \La$, where the right ideal $p \La$ is generated by all paths $q \notin I$ such that $q=pq'$ for some path $q'$. The situation is better when we concentrate on the quadratic monomial algebras, that is, the ideal $I$ is generated by paths of length 2.

In the rest of this remark, assume that $\La$ is a quadratic monomial algebra. By \cite[lemma 3.4]{CSZ}
for a perfect pair $(p, q)$ in $\La$, both $p$ and $q$ are necessarily arrows. In particular, a perfect path is an arrow and its relation cycle consists entirely of arrows. Then by using \cite[Theorem 4.1]{CSZ}, one may deduce that there is a bijection between perfect arrows in $\La$ and the set of isoclasses indecomposable non-projective Gorenstein projective $\La$-modules.

In \cite[Definition 5.1]{CSZ}, the relation quiver $\mathcal{R}_\Lambda$ of $\Lambda$ is defined. Its vertices are given by the arrows in $Q$, and there is an arrow $[\beta\alpha]:\alpha \rightarrow \beta$ if the head of $\beta$ is equal to the tail of $\alpha$ and $\beta\alpha$ lies in $I$. A component $\mathcal{D}$ of $\mathcal{C}_\Lambda$ is called perfect if it forms a basic cycle \cite[Definition 5.1]{CSZ}. From \cite[Lemma 5.3]{CSZ}, we observe that an arrow $\alpha$ is perfect if and only if the corresponding vertex in $\mathcal{C}_\Lambda$ belongs to a perfect component. As mentioned in the proof of \cite[Theorem 5.7]{CSZ}, for a perfect arrow $\alpha$ lying in the perfect component $\mathcal{D}$, its relation cycle is of the form $\alpha = \alpha_1, \alpha_2, \cdots, \alpha_{d_i}$ with $\alpha_{d_i+1} = \alpha$. Moreover, we have the equalities $\Omega_\Lambda(\alpha_i\Lambda) = \alpha_{i+1}\Lambda$ and $\Omega^{-1}_\Lambda(\alpha_i\Lambda) = \alpha_{i-1}\Lambda$. According to the above observation, we see that the set formed by all perfect components of the relation quiver $\mathcal{R}_\Lambda$ are in one-to-one correspondence with $\mathcal{C}(\Lambda)$, by sending a given perfect component $\mathcal{D}$ to the equivalence class $[\alpha\Lambda]$ in $\mathcal{C}(\Lambda)$, where $\alpha$ is a vertex of $\mathcal{D}$.

\end{remark}
\subsection{Singularity category of G-semisimple algebras}
Let $\CT$ be an additive category and $F:\CT \rt \CT$ an automorphism. Following \cite{K}, the orbit category $\CT/F$ has the same objects as $\CT$ and its morphisms from $X$ to $Y$ are $\oplus_{n \in \mathbb{Z}}\Hom_{\CT}(X, F^n(Y) ).$ Further, the singularity category $\mathbb{D}_{\text{sg}}(\La)$ of an algebra $\La$ is defined as the Verdier quotient of the bounded derived category $\mathbb{D}^{\text{b}}(\md$-$\La)$ by the subcategory of bounded complexes of projective modules.

\begin{proposition}\label{proposition 2.8}
Let $\La$ be a  G-semisimple algebra over an algebraic closed field $k$. If $\La$ is a Gorenstein algebra, then there is an equivalence of triangulated categories $$ \mathbb{D}_{\text{sg}}(\La)\simeq \prod_{[G]\in \mathcal{C}(\La)} \frac{\mathbb{D}^{\rm{b}}(\md\mbox{-}k)}{l(G)}, $$ where $\frac{\mathbb{D}^{\rm{b}}(\md\mbox{-} k)}{l(G)}$ denotes the triangulated orbit category of $\mathbb{D}^{\rm{b}}(\md\mbox{-}k)$ with respect to the functor $F=[l(G)]$, the $l(G)$-th power of the shift functor.
\end{proposition}
\begin{proof}
Since $\La$ is a Gorenstein algebra, by a fundamental result of Buchweitz and Happel, there is a triangle equivalence $\mathbb{D}_{sg}(\La)\simeq \ugpr$-$\La$. So it is enough to describe $\ugpr$-$\La$. Due to the lines before Remark \ref{remark 2.7}, any indecomposable non-projective object in $\gpr$-$\La$ is isomorphic to a module in an equivalence class in $\CC(\La)$.

For any equivalence class $[G] \in \mathcal{C}(\La)$, we have $\Omega^{-l(G)}(G')\simeq G'$ in $\ugpr$-$\La$, where $G' \in [G].$ One should note that the suspension functor $\Sigma$ on $\mathrm{mod}$-$\Lambda$ is the quasi-inverse of the syzygy functor $\Omega_{\Lambda}^{-1}$. On the other hand, by the proof of Proposition \ref{Omega-algebra}, one may deduce that for every indecomposable modules $X$ and $Y$ in $\ugpr$-$\La$, $\underline{\Hom}_{\La}(X, Y)=0$ if $X\ncong Y$, and each non-zero morphism $\underline{f}$ in $\underline{\Hom}_{\La}(X, Y)$ is an isomorphism if $\underline{X} \cong \underline{Y}.$ Hence, $\uEnd_{\La}(\underline{X})$ is a division algebra containing $k.$ As $k$ is algebraic closed, this implies that $\uEnd_{\La}(X)\simeq k$ as $k$-vector space (or more as algebras). Now the rest of the proof goes through as for \cite[Theorem 2.5(b)]{Ka}. So the proof is finished.
\end{proof}

The above result is inspired by the similar equivalence in \cite{Ka} for gentle algebras and for quadratic monomial algebras in \cite{CSZ}.

\begin{remark}\label{Remark 2.8}
\begin{itemize}
\item [$(1)$] The algebras of finite global
dimension are trivially G-semisimple algebras, since $\gpr$-$\La= \mathsf{proj}$-$\La$. Although, our results in this paper are not interesting to them.

\item[$(2)$] Assume that $\La = kQ/I$ is a quadratic monomial algebra. According to \cite[Theorem 5.7]{CSZ}, there is a triangle equivalence
$\ugpr$-$\La\simeq \CT_{d_1} \times \CT_{d_2} \times \cdots \times \CT_{d_m},$
where $\CT_{d_n}=(\md$-$k^n, \sigma_n)$ for some natural number $n,$ and auto-equivalence $\sigma_n: \md$-$ k^n \rt \md$-$ k^n.$ Now, since for each $n$, $\md$-$(\md$-$k^n)$ is semisimple, the above equivalence yields that a quadratic monomial algebra is a G-semisimple algebra. Especially, gentle algebras are G-semisimple algebras.
\item[$(3)$] Let $\La$ be a simple gluing algebra of $A$ and $B$, see the papers \cite{Lu1} and \cite{Lu2} for two different ways of defining gluing algebras. In these two papers, for both types of gluing, it is proved that $\ugpr$-$\La\simeq\ugpr$-$A \coprod \ugpr$-$B$. So, if $\md$-$ (\ugpr$-$A)$ and $\md$-$ (\ugpr$-$B)$ are both semisimple, then the same will be true for $\md$-$ (\ugpr$-$\La)$. This means that  simple gluing operation preserves the G-semisimplicity of algebras. For example, cluster-tilted algebras of type $\mathbb{A}_n$ and endomorphism algebras of maximal rigid objects of cluster tube $\mathcal{C}_n$ can be built as simple gluing of G-semisimple algebras, see \cite[ Corollaries 3.21 and 3.22]{Lu2}. Hence this kind of cluster-tilted algebras are G-semisimple.
\item [$(4)$] Assume that $\La$ and $\La'$ are derived equivalent algebras. According to Theorem 3.24 of \cite{AHV},
$\ugpr$-$\La\simeq \ugpr$-$\La'.$ This implies that$\Lambda$ and $\Lambda'$ are G-semisimple, simultaneously.
Hence G-semisimplicity of algebras are closed under derived equivalences.
\item [$(5)$] Assume that $\Lambda$ is an Artin algebra. According to \cite[Theorem 10.7]{AR1}, $\mathrm{mod}$-$(\underline{\mathrm{mod}}$-$\Lambda)$ is a semisimple abelian category if and only if $\Lambda$ is a Nakayama algebra with loewy length at most $2$. This implies that the self-injective Nakayama algebras $k[x]/(x^n)$, with $n>2$, are examples of $\cm$-finite algebras that are not G-semisimple algebras. Therefore, there exist $\cm$-finite algebras that are not G-semisimple.
\end{itemize}
\end{remark}
\subsection{1-Gorenstein G-semisimple algebras}
Recall that an algebra $\La$ is $1$-Gorenstein if the injective dimension of $\La$ as a left or right module (in this case one side is enough) is at most one.
In the following, we shall give a characterization of $1$-Gorenstein G-semisimple algebras in terms of the radical of their projective modules.
\begin{proposition}\label{RadGorp}
The following conditions are equivalent:
\begin{itemize}
\item[$(i)$] $\La$ is a $1$-Gorenstein G-semisimple algebra.
\item[$(ii)$]$\mathsf{rad} \La\oplus \La$ generates $\gpr$-$\La$, i.e., $\gpr$-$ \La=\mathsf{add}(\mathsf{rad} \La\oplus \La)$ .

\end{itemize}
\end{proposition}
\begin{proof}
$(i)\Rightarrow (ii)$. Since $\La$ is
$1$-Gorenstein,	$\Omega^1_{\La}(\md$-$\La)=\gpr$-$\La$. Let $Q$ be an arbitrary indecomposable projective $\La$-module. By Lemma \ref{Radical}, $\mathsf{rad} Q\hookrightarrow Q$ is a minimal right almost split morphism in $\gpr$-$\La$. Now suppose that $G$ is an indecomposable non-projective Gorenstein projective module. There exists an indecomposable non-projective Gorenstein projective module $G'$ such that $\Omega_{\La}(G')=G$. Since $\La$ is G-semisimple, we can have the almost split sequence $0 \rt G \st{f} \rt P \rt G' \rt 0$ in $\gpr$-$\La$. Now let $Q$ be an indecomposable direct summand of $P$. Since $f$ is a left minimal almost split morphism, there exists an irreducible morphism from $G$ to $Q$. As mentioned earlier, there also exists the right minimal almost split morphism $\mathsf{rad}Q\hookrightarrow Q$. Therefore, $G$ must be a direct summand of $\mathsf{rad}Q$, and hence a direct summand of $\mathsf{rad}\Lambda$, as required.

$(ii)\Rightarrow (i).$ Since $\mathsf{rad}\La$ is a Gorenstein projective module, the simple modules have Gorenstein projective dimension at most one. Then by induction on the length of modules, one may see that the global Gorenstein projective dimension of $\md$-$\La$ is at most one, equivalently, $\La$ is $1$-Gorenstein. In order to prove that $\La$ is a G-semisimple algebra, we first claim that there are no irreducible morphisms between two indecomposable non-projective Gorenstein projective modules. Assume on the contrary that there is an irreducible morphism $f: G \rt G'$ where $G$ and $G'$ are indecomposable non-projective Gorenstein projective modules. As $\gpr$-$\La$ is closed under submodule, the same as the module category case, we deduce that $f$ would be either a monomorphism or an epimorphism. First assume that $f$ is an epimorphism and $0\rt \tau_{\CG}G'\rt E\st{g}\rt G'\rt 0$ is an almost split sequence in $\gpr$-$\La$. Clearly, $\tau_{\CG}G'$ is an indecomposable non-projective Gorenstein projective module. So by our assumption, there is an indecomposable projective module $P$ such that
$\tau_{\CG}G'$ is a direct summand of $\mathsf{rad} P.$ Since by Lemma \ref{Radical}, the inclusion $\rad P\hookrightarrow P$ is a right minimal	almost split morphism in $\gpr\mbox{-}\La$, $P$ must be a direct summand of $E$. Thus, we have the irreducible morphism $g \mid : P \rightarrow G'$. Since $P$ is projective and $f$ is an epimorphism, there exists a morphism $h : P \rightarrow G$ such that $f \circ h = g \mid$. As $g\mid$ is an irreducible morphism, $h$ will be a split monomorphism or it is a split epimorphism. But both cases are impossible, as modules $G, G'$ are indecomposable non-projective.
The established claim gives us that for any indecomposable non-projective Gorenstein projective module $X$, if $g: Y \rt X $ is a right minimal almost split morphism in $\gpr$-$\La$, then $Y$ must be a projective module. Note that by $(ii)$ the subcategory $\gpr$-$\La$ is of finite representation type, and so it has almost split sequences. Suppose that an indecomposable non-projective representable functor $(-, \underline{G})$ is given. We may assume that $G$ is an indecomposable non-projective object of $\gpr$-$\La$. Take an almost split sequence $\la: 0 \rt \tau_{\CG}G\rt E\rt G\rt 0$ in $\gpr$-$\La.$ As we have observed, $E$ is projective. Thus, applying the Yoneda functor on $\la$ gives us the following sequence in $\md\mbox{-}(\gpr$-$\La)$
$$ 0 \rt (-, \tau_{\CG}G)\rt (-, E)\rt (-, G)\rt (-, \underline{G})\rt 0.$$Since $\la$ is almost split in $\gpr$-$\La$, the latter sequence implies that $(-, \underline{G})$ is simple in $\md$-$(\gpr$-$\La).$ Consequently, as any simple functor in $\md$-$(\ugpr$-$\La)$ arising in this way, induced by an almost split sequence, see e.g. \cite[Chapter 2]{Au3}, we obtain that $\md$-$(\ugpr$-$\La)$ is a semisimple abelian category. Hence $\La$ is a G-semisimple algebra, and so, the proof is finished.
\end{proof}
We close this section by presenting some examples of $1$-Gorenstein G-semisimple algebras.
\begin{example}\label{123}
\begin{itemize}
\item [$(i)$] Clearly, hereditary algebras are $1$-Gorenstein G-semisimple algebras. So, $1$-Gorenstein G-semisimple algebras can be considered as a generalization of hereditary algebras.
\item[$(ii)$] Recently, Ming Lu and Bin Zhu provided a criterion in \cite{LZ} for determining which monomial algebras are $1$-Gorenstein algebras. With this criterion available, one can search among quadratic monomial algebras to find $1$-Gorenstein G-semisimple algebras. In particular, one can focus on gentle algebras to determine which of them are $1$-Gorenstein, as studied in \cite{CL}.
\item[$(iii)$] The cluster-tilted algebras, defined in \cite{BMR3} and \cite{BMD4}, are an important class of $1$-Gorenstein algebras. So, as provided below one example, among cluster-tilted algebras, we can find some examples of $1$-Gorenstein G-semisimple algebras. For example, the cluster-tilted algebras of type $A$ are gentle algebras. Therefore, in this case, we are dealing with $1$-Gorenstein G-semisimple algebras.
For the other types $D$ and $E$, in \cite{CGL} the singularity categories of cluster-tilted algebras are described by the stable categories of some self-injective algebras. In particular, those cluster-tilted algebras of type $D$ and $E$ which are singularity equivalent to the self-injective Nakayama algebra $\La(3, 2)$, the Nakayama algebra with cycle quiver with $3$ vertices modulo the ideal generated by the paths of length $2$, are other examples of G-semisimple algebras. The following cluster-tilted algebra given by the quiver
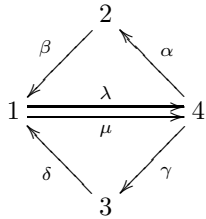
\begin{equation*}
\xymatrix{& 2 \ar[ld]_{\beta} \\
1\ar@<2pt>[rr]^{\lambda}\ar@<-2pt>[rr]_{\mu}& & 4 \ar[lu]_{\alpha}\ar[ld]^{\gamma}\\
&3\ar[lu]^{\delta}}
\end{equation*}bound by the quadratic monomial relations $\alpha \beta=0, \ \gamma \delta=0, \ \delta \lambda=0, \ \lambda \gamma=0, \ \beta \mu=0,$ and $\mu \alpha=0 $ is also a $1$-Gorenstein G-semisimple algebra.
\end{itemize}
\end{example}

\section{the path algebras of  G-semisimple algebras}
Assume that $\CQ$ is an acyclic quiver. This section is devoted to show that  for any  G-semisimple finite dimensional algebra  $\La$ over an algebraic closed field $k$,  there exists  an epivalence  from the stable category of Gorenstein projective representations of $\CQ$ by $\La$-modules, $\ugpr(\CQ, \La)$, to the category of representations of $\CQ$ by the stable category of Gorenstein projective $\La$-modules, ${\rm rep}(\CQ, \ugpr$-$\Lambda)$. As a result, we obtain the Gabriel-style classification of Gorenstein projective representations over G-semisimple algebras. Let us inaugurate this section by adapting some necessary definitions.

Let $\CQ=(V, E)$ be a quiver, i.e., a directed graph with the set of vertices $V$ and the set of arrows $ E$ endowed with a couple of functions
$s$ and $t$ which assign, respectively, to any arrow $a$ of its origin and terminal vertices. Let $\CC$ be an additive category. A representation $\CX$ of $\CQ$ by $\CC$ is obtained by associating to any vertex $v $ an object $\CX_v$ in $\CC$ and to any arrow $a:v\rt w$ a morphism $\CX_a:\CX_v\rt \CX_w$ in $\CC.$ Such a representation is denoted as $\CX=(\CX_v, \CX_a)_{v \in V, a \in E }$, or simply $\CX=(\CX_v, \CX_a)$. If $\CX$ and $\CY$ are two representations of $\CQ$, then a morphism $f:\CX \rt \CY$ is determined by a
family $f=(f_v)_{v \in V}$ of $\CC$ in such a way that for any arrow $a:v\rt w$,
the commutativity condition $f_w\CX_a=\CY_af_v $ holds. The representations of $\CQ$ by $\CC$ and the morphisms between them form a category that
is denoted by ${\rm rep}(\CQ, \CC)$. In the case $\CC=\md$-$\La$, then ${\rm rep}(\CQ, \md$-$\La)$ is simply denoted by ${\rm rep}(\CQ, \La)$.

For a subquiver $\CQ' \subseteq \CQ$, the restriction functor
$$e^{\CQ'}:=e^{\CQ'}_{\CC}:{\rm rep}(\CQ, \CC) \rt {\rm rep}(\CQ', \CC),$$
is defined to restrict any representation of $\CQ$ to the vertices of $\CQ'.$ In particular, for any vertex $v \in V$ of quiver, let $\CQ'=\{v\}$ be the subquiver consisting only of $v$.
Let $e^v:{\rm rep}(\CQ, \La) \rt \md$-$\La$ be the evaluation
functor defined by $e^v (\CX ) = \CX_v$, for any representation $\CX$. It is proved in \cite{EH} that $e^v$ has a left adjoint $e_{\la}:\md$-$\La \rt {\rm rep}(\CQ, \La)$ given by $e^v_{\la}(M)_w=\oplus_{\CQ(v, w)} M$, where
$\CQ(v, w)$ denotes the set of paths starting in $v$ and terminating in $w$, and by the natural injection $e^v_{\la}(a):\oplus_{\CQ(v, w_1)} M \rt \oplus_{\CQ(v, w_2)}M$ for an arrow $a:w_1\rt w_2$. Moreover, it is shown that $e^v$ admits also a right adjoint $e^v_{\rho}$, defined by $e^v_{\rho}(M)_w=\prod_{\CQ(w, v)} M $.

If $p = a_m \cdots a_1$ is a path with $s(p)=v$ and $t(a_m)=w$ and $\CX$ a representation in ${\rm rep}(\CQ, \La)$, then the morphism $\CX_p=\CX(a_m)\cdots \CX(a_1):\CX_v\rt \CX_w$ is defined by the composition of maps on arrows $a_p$ in the path $p$.

Let $v$ be a vertex in $\CQ$ and consider the adjoint pair $(e^v_{\la}, e^v)$. For every $M \in \md$-$\La$ and $\CX$ in ${\rm rep}(\CQ, \La)$, the natural isomorphism is given by:
\begin{equation}\label{adjuction-iso} \Theta_{M, \CX}: \Hom_{\La}(M, \CX_v) \rt \Hom_{\CQ}(e^v_{\la}(M), \CX) \end{equation}
$$f:M \rt \CX_v \mapsto \tilde{f}=(\tilde{f}_v)_{v \in V}, $$
$\tilde{f}_w:\oplus_{\CQ(v, w)} M \rt \CX_w$ is defined by sending $(a_{p})_{p \in \CQ(v, w)} \in \oplus_{ \CQ(v, w)} M$ to $\Sigma_{p \in \CQ(v, w)}\CX_p\circ f$.

It is known that a representation $\CX$ of acyclic quiver $\CQ$ is projective if and only if for any vertex $v$:
\begin{itemize}
\item [$(i)$]$\CX_v$ is a projective $\La$-module. \item [$(ii)$]
The $\La$-morphism $\CX^v: \oplus_{t(a)=v}\CX_{s(a)} \rt \CX_v$ is a splitting monomorphism.
\end{itemize}
Here $\CX^v$ denote the $\La$-homorphism $\oplus_{t(a)}\CX_{s(a)} \rt \CX_v$ which defined by sending $(x_{s(a)})_{t(a)=v}$ to $\Sigma_{t(a)=v}\CX_a(x_{s(a)})$, see \cite[Theorem 3.1]{EE}. A quiver $\CQ$ is called {\em acyclic} if it does not contain any directed cycles.\\ Assume that $\gpr(\CQ, \La)$ is the subcategory of ${\rm rep}(\CQ, \La)$ consisting of all Gorenstein projective representations. The following result gives a local characterization of Gorenstein projective representations of an acyclic quiver.

\begin{theorem}(\cite[ Theorem 3.5.1 ]{EHS} or \cite[Theorem 5.1]{LuZ}) \label{Gorenproj charecte}
Let $\CQ$ be an acyclic quiver and $\CX$ a representation in $\rm{rep}(\CQ, \La)$. Then $\CX$ is in $\gpr(\CQ, \La)$ if and only if
\begin{itemize}
\item [$(i)$] For each vertex $v,$ $\CX_v$ is a Gorenstein projective $\La$-module;
\item[$(ii)$] For each vertex $v,$ the $\La$-morphism $\CX^v: \oplus_{t(a)=v} \CX_{s(a)} \rt \CX_v$ is a monomorphism whose cokernel is Gorenstein projective.
\end{itemize}
\end{theorem}
{It is worth pointing out that, since $\CQ$ is an acyclic quiver, one may easily see that the (Gorenstein) projectiveity of vertices in the local characterization of representations, the condition $(i)$, is redundant.}

Assume that $\mathsf{mono}(\CQ, \La)$ is the (separated) monomorphism category, consisting of all representations $\CX$ for which the $\La$-homorphism $\CX^v:\oplus_{t(a)=v}\CX_{s(a)} \rt \CX_v$, for any vertex $v$, is monomorphism. It is evident that $\mathsf{mono}(\CQ, \La)$ is an exact full subcategory of ${\rm rep}(\CQ, \La)$. Assume that $\overline{\mathsf{mono}}(\CQ, \La)$ and $\overline{\md}\mbox{-}\La$ are the respective injectively stable categories. Let $\CQ$ be an acyclic quiver. Gao et.al. constructed an epivalence functor $F: \overline{\mathsf{mono}}(\CQ, \La)\lrt {\rm rep}(\CQ, \overline{\md}\mbox{-}\La)$, and moreover, if $\CQ$ has at least one arrow, then $F$ is equivalence if and only if $\La$ is a hereditary algebra, see \cite[Theorems 5.6 and 5.14]{GKKP}. Recall that a functor is said to be epivalence (or representation equivalence in the sense of Auslander) if it is full and dense which reflects isomorphisms.
Motivated by the functor $F$, we define the functor $\Psi$ as follows.

Assume that $\pi:\gpr$-$\La \rt \ugpr$-$\La$ is the canonical functor, where $\ugpr$-$\La$ is the stable category of Gorenstein projective $\La$-modules. For simplicity, for any morphism $f$ in $\gpr$-$\La$, let $\underline{f}:=\pi(f)$, and to stress that we consider a Gorenstein projective module $G$ in the stable category $\ugpr$-$\Lambda$, we denote it by $\underline{G}$. We define the following functor
$$\Psi:=\Psi^{\CQ}:\gpr(\CQ, \La)\rt {\rm rep}(\CQ, \ugpr\mbox{-}\La)$$
$$\CX=(\CX_v)_{v \in V}\mapsto \Psi(\CX):=(\underline{\CX_v})_{v \in V}$$
$$f=(f_v)_{v \in V}\colon \CX \rt \CY \mapsto \Psi(f)=(\underline{f_v})_{v \in V}$$
One should note that, since each vertex of a projective presentation, is a projective $\La$-module, the functor $\Psi$ induces the functor
$$\underline{\Psi}:\ugpr(\CQ, \La)\rt {\rm rep}(\CQ, \ugpr\mbox{-} \La),
$$
where $\ugpr(\CQ, \La)$ is the stable category of $\gpr(\CQ, \La)$ modulo projective representations.
The aim of this section is to prove that $\underline{\Psi}$ is an epivalence, extending the above epivalence within the framework of exact categories. Precisely, we prove the following result.
\begin{theorem}\label{main}Let $\CQ$ be an acyclic quiver. Then $\underline{\Psi}$ is full and reflects isomorphisms. Moreover, $\underline{\Psi}$ is dense, whenever $\La$ is a G-semisimple finite dimensional algebra over an algebraic closed field $k$.
\end{theorem}

The proof of the above theorem is divided into three steps in the sequel.
For an acyclic quiver $\CQ$, there is a partition of $V$, which is defined inductively. Set
$$V_1:=V_1(\CQ)=\{ v \in E \colon \nexists \ a \in E \ \text{ such that } \ s(a)= v\}$$
Suppose that $i>1$ is a number and that we have defined $V_j$ for every $j<i.$ Let
$$V_i:=V_i(\CQ)=\{ v \in V \setminus \cup_{j <i}V_j \colon \nexists \ a \in E\setminus \{a \colon t(a) \in \cup_{j<i} V_j\} \ \text{such that} \ s(a)=v\}$$

For an acyclic quiver $\CQ$, there exists an integer $m$ such that $V=\cup_{i \leq m} V_i$. Denote by $\mu(\CQ)$ the least ordinal number for which $V$ can be written in this way. For a number $n$, let $\CQ^n$ be the
subquiver of $\CQ$ whose set of vertices is $V^n=\cup_{i \leq n} V_i$.

Due to the characterization of Gorenstein projective presentations and the construction of $V_i$, it can be observed that $e^{\CQ^n}(\gpr(\CQ, \La))=\gpr(\CQ^n, \La)$. This indicates that the restriction functor $e^{\CQ^n}$ preserves Gorenstein projective representations. Additionally, it is straightforward to verify the existence of the commutative diagram
\begin{equation}
\xymatrix{\gpr(\CQ, \La) \ar[d]_{e^{\CQ^n}} \ar[r]^{\Psi^{\CQ}}& {\rm rep}(\CQ, \ugpr\mbox{-} \Lambda) \ar[d]^{e^{\CQ^n}} \\ \gpr(\CQ^n, \La) \ar[r]^{\Psi^{\CQ^n} } & {\rm rep}(\CQ^n, \ugpr \mbox{-} \Lambda). }
\end{equation}

For simplicity, we denote $\underline{\CX}:=\Psi(\CX)$ for every representation $\CX \in\gpr(\CQ, \La).$

\begin{proposition}\label{full}
The functor $\Psi$ is full. In particular, $\underline{\Psi}$ is also full.
\end{proposition}
\begin{proof}
Let $\CX$ and $\CY$ be representations in $\gpr(\CQ, \La).$ Assume $h=(h_v)_{v \in V}$ is a morphism from $\underline{\CX}$ to $\underline{\CY}$ in ${\rm rep}(\CQ, \ugpr\mbox{-} \La)$. We construct inductively a morphism $g=(g_v)_{v \in V}$ in $\gpr(\CQ, \La)$ such that $\Psi(g)=h.$ For each vertex $v \in V_1$, we can choose $g_v:\CX_v\rt \CY_v$ in $\gpr\mbox{-}\La$ such that $h_v=\underline{g_v}$. Now assume that $n >1$ and we have constructed $g_v$ for every $v \in \cup_{j \leq n} V_j$ such that $\Psi^{\CQ^n}(g^n)=h^n$, where $g^n=(g_v)_{v \in V^n}$ and $h^n:=e^{\CQ^n}(h)$. Take a vertex $v \in V_{n+1}.$
Since $h$ is a morphism in ${\rm rep}(\CQ, \ugpr \mbox{-} \La)$ and $h_w=\underline{g_w}$ for every $w \in \cup_{j \leq n} V_j$, we have the following commutative diagram in $\md\mbox{-}\La$
\begin{equation}\label{Diagram:Lem-full}
\xymatrix{ \oplus_{t(a)=v}\underline{\CX}_{s(a)} \ar[d]_{\oplus_{t(a)=v}\underline{g_{s(a)}}} \ar@{=}[rr]&& \oplus_{t(a)=v}\underline{\CX}_{s(a)} \ar[d]_{\oplus_{t(a)=v}h_{s(a)}} \ar[r]^{\underline{\CX}^v}& \underline{\CX}_v \ar[d]^{h_v} \\ \oplus_{t(a)=v}\underline{\CX}_{s(a)} \ar@{=}[rr] && \oplus_{t(a)=v}\underline{\CY}_{s(a)} \ar[r]^{\underline{\CY}^v} & \underline{\CY}_v }
\end{equation}
Since for every vertex $w$, the module $\CX_w$ is a Gorenstein projective module, there exists a monomorphism $\CX_w\st{\alpha_w}\rt P_w$ with $P_w$ a projective module and its cokernel ${\mathsf{cok}}(\alpha_w)$ a Gorenstein projective module. Hence, by taking direct sum over vertices $\alpha_{s(a)}$ with $t(a)=v$, we get the monomorphism $\oplus_{t(a)=v}\alpha_a: \oplus_{t(a)=v} \CX_{s(a)}\rt \oplus_{t(a)=v}P_{s(a)}$ with the cokernel ${\mathsf{cok}}(\oplus_{t(a)=v}\alpha_a)=\oplus_{t(a)=v}{\mathsf{cok}}(\alpha_v)$. According to the diagram \ref{Diagram:Lem-full}, $h_v\underline{\CX}^v= \underline{\CY}^v\oplus_{t(a)=v}\underline{g_{s(a)}}$ in $\ugpr\mbox{-} \La$. Let $h_v=\underline{h'_v}$. Hence, $h'_v\CX^v-\CY^v\oplus_{t(a)=v}g_{s(a)}$ factors through a projective module. As $\oplus_{t(a)=v}\alpha_v$ is a left ${\mathsf{prj}}\mbox{-}\La$-approximation of $\oplus_{t(a)=v} \CX_{s(a)}$, $h_v\CX^v-\CY^v\oplus_{t(a)=v}g_{s(a)}$ factors through $\oplus_{t(a)=v}P_{s(a)}$ namely, there is the following factorization

$$\xymatrix@1{ 0\ar[r] & \oplus_{t(a)=v} \CX_{s(a)}\ar[d]_{h'_v\CX^v-\CY^v\oplus_{t(a)=v}g_{s(a)}}
\ar[rr]^{\oplus_{t(a)=v}\alpha_{s(a)}}
&& \oplus_{t(a)=v}P_v\ar@{.>}[dll]^{d} \\ & \CY_v&}$$
By Theorem \ref{Gorenproj charecte}, the cokernel of $\CX^v$ is a Gorenstein projective module, which yields the existence of the following factorization:
$$\xymatrix@1{ 0\ar[r] & \oplus_{t(a)=v} \CX_{s(a)}\ar[d]_{\oplus_{t(a)=v}\alpha_{s(a)}}
\ar[r]^>>>>{\CX^v}
& \CX_v \ar@{.>}[dl]^{r} \\ & \oplus_{t(a)=v}P_v&}$$
{Combining the above two diagrams, gives us } the following commutative diagram:
\begin{equation*}
\xymatrix{ & \oplus_{t(a)=v}\CX_{s(a)} \ar[d]_{\oplus_{t(a)=v}g_{s(a)}} \ar[r]^{\CX^v}& \CX_v \ar[d]^{h'_v-dr} \\ & \oplus_{t(a)=v}\CY_{s(a)} \ar[r]^{\CY^v} & \CY_v }
\end{equation*}
Set $g_v:=h'_v-dr$. Note that $\underline{h'_v-dr}=\underline{h'_v}- \underline{dr}=\underline{h'_v}=h_v$. By repeating this process for every vertex $v$ in $V_{n+1}$, one obtains a morphism $g^{n+1}=(g_v)_{v \in V^{n+1}}$ which satisfies the required conditions. So the proof is complete.
\end{proof}

Let $\CX$ be a representation in ${\rm rep}(\CQ, \La)$. There exists an short exact sequence 
\begin{equation}\label{eq. short}
0 \rt \oplus_{a \in E}e^{t(a)}_{\la}(\CX_{s(a)})\st{f_{\CX}}\rt \oplus_{v \in V}e^v_{\la}(\CX_v)\st{g_{\CX}}\rt \CX\rt 0
\end{equation}
where $f_{\CX}=(f_{\CX,w})_{w\in V}$ and $g_{\CX}=(g_{\CX, w})_{w \in V}$ are defined as follows:
\begin{itemize}
   \item For a vertex $w \in V$, let $\CQ(w)$ denote all paths ending at $w$, including the trivial path $e_w$. Additionally, define $\CQ'(w) = \CQ(w) - \{e_w\}$;
\item Any element of $(\oplus_{a \in V}e^v_{\lambda}(\CX_v))_w$ can be represented as $(\beta_q)_{q \in \CQ(w)}$, where $\beta_q \in \CX_{s(q)}$. Similarly, since $\CQ$ is acyclic, any element in $(\oplus_{a \in E}e^{t(a)}_{\lambda}(\CX_{s(a)}))_w$ can be viewed as $(\beta_q)_{q \in \CQ'(w)}$ with $\beta_q \in \CX_{s(q)}$;

\item For an  element $\beta=(\beta_q)_{q \in \CQ(w)}$ in  $(\oplus_{a \in V}e^v_{\lambda}(\CX_v))_w$, define $g_{\CX, w}(\beta):=\Sigma_{q \in \CQ(w)} \CX_{q}(\beta_q);$ 
\item  Consider   $\beta=(\beta_q)_{q \in \CQ'(w)}$ as  an element  in  $(\oplus_{a \in E}e^{t(a)}_{\lambda}(\CX_{s(a)}))_w$. For any $q \in \CQ'(w)$, define $\overline{\beta}_q:=(\beta'_p)_{p \in \CQ'(w)}$, where $\beta'_q=\beta_q$, and $\beta'_p=0$ if $p\neq q.$ To define $f_{\CX, w}(\beta)$, it suffices  to define $f_{\CX, w}(\overline{\beta}_q)$ for any $q \in \CQ'(w)$. Thus, let $q=ha \in \CQ'(w)$ with arrow $a$, and  define $f_{\CX, w}(\overline{\beta}_q):=(\alpha_{p })_{p \in \CQ(w)}$, where $\alpha_q=\beta_q$ and $\alpha_h:=-\CX_a(\beta_q)$, and  all other components are zero.    
\end{itemize}

The above short exact sequence behave  canonically. This meant that if there exists a morphism $\phi=(\phi_v)_{v\in V}:\CX\rt \mathcal{Y}$ in ${\rm rep}(\CQ, \La)$, we have the following commutative diagram:

\begin{equation} \label{eq. diagram}   
\xymatrix{
0 \ar[r] &\oplus_{a \in E}e^{t(a)}_{\la}(\CX_{s(a)}) \ar[d]^{\phi_1} \ar[r]^{f_{\CX}} & \oplus_{v \in V}e^v_{\la}(\CX_v) \ar[d]^{\phi_0}
\ar[r]^<<<{g_{\CX}} &\CX \ar[d]^{\phi}\ar[r] & 0&\\
0 \ar[r] & \oplus_{a \in E}e^{t(a)}_{\la}(\CY_{s(a)}) \ar[r]^{f_{\CY}} & \oplus_{v \in V}e^v_{\la}(\CY_v)
\ar[r]^<<<{g_{\CY}} & \CY\ar[r] &0.& \\ } 
\end{equation}
Here,  the matrix presentations  of  $\phi_0$ and $\phi_1$ with respect to the corresponding factorizations are diagonal matrices $[e^v_{\la}(\phi_v)]_{v \in V}$ and $[e^v_{\la}(\phi_{s(a)})]_{a \in E}$, respectively.

\begin{proposition}\label{faithful}
The functor $\underline{\Psi}$ reflects isomorphisms.
\end{proposition}
\begin{proof}
Let $\phi'=(\phi'_v)_{v\in V}:\CX\rt \CY$ be a morphism in $\ugpr(\CQ, \La)$ such that $\underline{\Psi}(\phi')$ is  an isomorphism in ${\rm rep}(\CQ, \ugpr\mbox{-} \La)$. Assume $\phi=(\phi_v)_{v\in V}:\CX\rt \CY$ is a morphism in $\gpr(\CQ, \La)$ such that after applying the canonical functor from $\gpr(\CQ, \La)$ to $\ugpr(\CQ, \La)$,  we obtain $\phi'$. In particular, for every vertex $v$, $\pi(\phi_v)=\underline{\phi_v}=\phi'_v.$   Applying \ref{eq. short} for the representations $\CX$ and $\CY$, we derive the short exact sequences 

$$\epsilon_1: \  \ 0 \rt \oplus_{a \in E}e^{t(a)}_{\la}(\CX_{s(a)})\st{f_{\CX}}\rt \oplus_{v \in V}e^v_{\la}(\CX_v)\st{g_{\CX}}\rt \CX\rt 0$$
and 
$$\epsilon_2: \ \  0 \rt \oplus_{a \in E}e^{t(a)}_{\la}(\CY_{s(a)})\st{f_{\CY}}\rt \oplus_{v \in V}e^v_{\la}(\CY_v)\st{g_{\CY}}\rt \CY\rt 0 $$
in the exact category $\gpr(\CQ, \La)$. For a given  vertex $v$ and Gorenstein projective module $G$, one may apply Theorem \ref{Gorenproj charecte} and deduce  that $e^v_{\la}(G)$ is a Gorenstein projective representation.

Due to the  construction of triangles in the triangulated category $\gpr(\CQ, \La)$, the short exact sequences $\epsilon_1$ and $\epsilon_2$ induces respectively  the following triangles:

$$ \oplus_{a \in E}e^{t(a)}_{\la}(\CX_{s(a)})\st{\underline{f_{\CX}}} \rt  \oplus_{v \in V}e^v_{\la}(\CX_v) \st{\underline{g_{\CX}}}\rt \CX \rightsquigarrow $$
and 
$$ \oplus_{a \in E}e^{t(a)}_{\la}(\CY_{s(a)})\st{\underline{f_{\CY}}} \rt  \oplus_{v \in V}e^v_{\la}(\CY_v) \st{\underline{g_{\CY}}}\rt \CY \rightsquigarrow. $$
Here, ``underline'' indicates the  residue class of the given map in the stable category   $\ugpr(\CQ, \La)$.  By applying \ref{eq. diagram} to the morphism $\phi:\CX \rt \CY$, we obtain a commutative diagram in $\gpr(\CQ, \La)$ as given in \ref{eq. diagram}. Based on the above, the resulting commutative diagram induces the following commutative diagram in $\ugpr(\CQ, \La)$:

\begin{equation}\label{eq.diam2}
\xymatrix{
 \oplus_{a \in E}e^{t(a)}_{\la}(\CX_{s(a)}) \ar[d]^{\underline{\phi_1}} \ar[r]^{\underline{f_{\CX}} } & \oplus_{v \in V}e^v_{\la}(\CX_v) \ar[d]^{\underline{\phi_0} }
\ar[r]^<<<{\underline{g_{\CX}}} &\CX \ar[d]^{\phi'}\rightsquigarrow& \\
 \oplus_{a \in E}e^{t(a)}_{\la}(\CY_{s(a)}) \ar[r]^{\underline{f_{\CY}}} & \oplus_{v \in V}e^v_{\la}(\CY_v)
\ar[r]^<<<{\underline{g_{\CY}}} & \CY  \rightsquigarrow & \\ } 
\end{equation}
with the triangle rows. Note that $\phi'=\underline{\phi}$, and especially the matrix presentations  of  $\underline{\phi_0}$ and $\underline{\phi_1} $ with respect to the corresponding factorization are the  diagonal matrices $[e^v_{\la}(\phi'_v)]_{v \in V}$ and $[e^v_{\la}(\phi'_{s(a)})]_{a \in E}$, respectively. Since  $\underline{\Psi}(\phi')$ is  an isomorphism in ${\rm rep}(\CQ, \ugpr\mbox{-} \La)$, this follows that for every vertex $v \in V$, $\phi'_v$ is an isomorphism in $\ugpr\mbox{-} \La$. Equivalently, for every vertex $v \in V$, $e^v_{\la}(\phi'_v)$ is an isomorphism in $\ugpr\mbox{-} \La$. Hence, $\underline{\phi_0}$ and $ \underline{\phi_1}$ are isomorphisms in $\ugpr(\CQ, \La)$. In view of Diagram  \ref{eq.diam2}, we obtain  $\phi'$ is an isomorphism, as desired.
\end{proof}

\begin{remark}\label{dd}
{The density of the functor $\underline{\Psi}$ relies on the fact that if $\La$ is a G-semisimple finite dimensional algebra over an algebraic closed field $k$, then for}
every indecomposable modules $\underline{X}$ and $\underline{Y}$ in $\ugpr\mbox{-}\La$, $\underline{\Hom}_{\La}(\underline{X}, \underline{Y})=0$ if $\underline{X}\ncong \underline{Y}$, and each non-zero morphism $\underline{f}$ in $\underline{\Hom}_{\La}(\underline{X}, \underline{Y})$ is an isomorphism if $\underline{X} \cong \underline{Y}.$ In particular, $\underline{\rm{End}}_{\La}(\underline{X})\simeq k$, i.e. every morphism $f \in \underline{\rm{End}}_{\La}(\underline{X})$ can be written as $f=a1_{\underline{X}}$ for some $a \in k.$ This observation was taken into account during the proof of Proposition \ref{proposition 2.8}.
\end{remark}

{Assume that $G$ is a Gorenstein projective module. Then it fits into a short exact sequence,}
$$0 \rightarrow G \xrightarrow{\gamma_G} P(G) \xrightarrow{\zeta_G} C(G) \rightarrow 0,$$
where $\zeta_G$ represents the projective cover. Consequently, $\gamma_G$ serves as a minimal left ${\mathsf{ prj}}\mbox{-}\La$-approximation. In the special case when $G$ is a projective module, we set $P(G)=G$ and $\gamma_G$ as the identity map.

{{\bf Proof of Theorem \ref{main}.\\}}
{According to Propositions \ref{full} and \ref{faithful}, we only need to show the second assertion. So assume that $\La$ is G-semisimple finite dimensional algebra over an algebraic closed field $k$. We have to show that the functor $\underline{\Psi} $ is dense. To do this, take } a representation $\CG=(\CG_v, f_a)_{v \in V, a \in E}$ in ${\rm rep}(\CQ, \ugpr\mbox{-}\La)$. Our goal is to construct a Gorenstein projective representation $\CH=(\CH_v, h_a)_{v \in V, a \in E}$ inductively such that $\Psi(\CH)\simeq \CG.$ By Proposition \ref{CM-finite}, $\La$ is a $\cm$-finite algebra. Assume that $\{G_1, \cdots, G_n\}$ is a complete set of pairwise non-isomorphic indecomposable Gorenstein projective non-projective modules. Without loss generality, one may assume that for any $v \in V$, $\CG_v=\CG_1^v\oplus \cdots \oplus \CG^v_{\alpha_v}$, which is a decomposition of $\CG_v$ into indecomposable modules $G_i$ in $\gpr\mbox{-}\La$. For every $v \in V^1$, let $\CH_v=\CG_v$. Assume that we have constructed $\CH_v$ and $\CH_a$ respectively for every vertex $v$ and arrow $a$ in $\CQ^n$ {such that for any vertex $v\in V^n$}:
\begin{itemize}
\item [$(i)$] $\CH_v=\CG_v\oplus P_v$, where $P_v$ is a projective module;
\item [$(ii)$] $\pi(\CH_v)=\underline{\CH}_v=\CG_v$, and for every $a \in E,$ $\pi(\CH_a)=\underline{\CH}_a=\CG_a$;
\item [$(iii)$] ${\mathsf{cok}} (\CH^v)$ is a Gorenstein projective module, and $\CH^v$ is injective.
\end{itemize}

For a given vertex $v \in V^{n+1}$, define $\CH_v=\CG_v \oplus_{t(a)=v}P(\CH_{s(a)})$. Assume that $\alpha_1, \cdots, \alpha_m$ are all the arrows with terminal vertex $v$

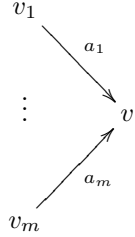
\[\xymatrix{v_{1}\ar[rd]^{a_{1}} & & \\ \vdots & v & \\ v_{m}\ar[ru]_{a_{m}} & & }\]
Clearly, all vertices $v_i$ lie in $V^{n}$. Note that by the inductive hypothesis, we have already defined $\CH_{v_i}$. Let $1 \leq d \leq m$ and $s(a_d)=v_d.$ By the representation $\CG$, we have the morphism $\CG_{a_d}:\CG_{v_d}\rt \CG_v$. According to the decompositions $\CG_{v}=\CG_1^v\oplus \cdots \oplus \CG^v_{\alpha_v}$ and $\CG_{v_d}=\CG_1^{v_d}\oplus \cdots \oplus \CG^{v_d}_{\alpha_{v_d}}$ in $\gpr\mbox{-}\La$, one has the matrix presentation $[(\CG_{a_d})_{ij}:\CG_i^{v_d}\rt \CG_j^{v}]_{i \in I, j \in J}$ of $\CG_{a_d}$, where $I=\{1, \cdots, \alpha_{v_d}\}$ and $J=\{1, \cdots, \alpha_{v}\}$. By Remark \ref{dd}, $(\CG_v)_{ij}=0$ or $(\CG_v)_{ij}=s_{d, i, j}1_{\underline{\CG^v_j}}$, for some non-zero element $s_{d, i, j} \in k$, in $ \ugpr\mbox{-}\La$. One should note that in the latter case, $\CG_i^{v_d}$ must be equal to $\CG_j^{v}$, and indeed $1_{1_{\underline{\CG^v_j}}}=\pi(1_{\CG^v_j})$, where $1_{\CG^v_j}$ is the identity on $\CG^v_j$ in $\gpr\mbox{-}\La$.

Define $h'_{a_d}:\CG_{v_d}=\CG_1^{v_d}\oplus \cdots \oplus \CG^{v_d}_{\alpha_{v_d}}\rt \CG_{v}=\CG_1^v\oplus \cdots \oplus \CG^v_{\alpha_v}$ by the matrix factorization $[(h'_{a_d})_{ij}]_{i \in I, j \in J}$ which is defined as follows:
\begin{itemize}
\item [$(1)$] $(h'_{a_d})_{ij}=0$, if $(\CG_{a_d})_{ij}=0$
\item [$(2)$] $(h'_{a_d})_{ij}=s_{d, i, j}1_{\CG^v_j}$, if $(\CG_{a_d})_{ij}\neq 0$
\end{itemize}

Now, we are ready to define $\CH_{a_d}:\CH_{v_d}\rt \CH_v$. It is defined as follows:
\begin{itemize}
\item [$(1)$] It has no image in $P(\CH_{s(a)})$ with $s(a)\neq v_d;$
\item [$(2)$] Its projection to $\CG_v \oplus P(\CH_{v_d})$ has the matrix factorization $[h'_{a_d}~~\gamma_{\CH_{v_d}}]$.
\end{itemize}
Repeat the same process to define $\CH_{a_i}$ for every $1 \leq i \leq m$. We will show that $\CH^v$ is injective with a Gorenstein projective cokernel.\\
Assume that $J_0$ is a subset of $J=\{1, \cdots, \alpha_v\}$ consisting of all $i$ such that for every $1 \leq d \leq m$, the morphism $\CG_{a_d}$ has no image in $\CG^v_i.$ We define $l$ as the composition of two morphisms: $\oplus_{t(a)=v} \CH_{s(a)} \rt{\CH^v}\rt \CH_v\rt \oplus_{i \in I\setminus I_0}\CG^v_{i} $, where the second morphism is the projective map. By the construction, $l$ is surjective. Consider the following pull-back diagram:
$$\xymatrix{&	&0\ar[d]&0 \ar[d]& \\	0\ar[r]& \Ker(l) \ar@{=}[d] \ar[r] & \oplus_{t(a)=v}\CH_{s(a)}\ar[d]^{\oplus_{t(a)=v}\gamma_{\CH_{{s(a)}}}}
\ar[r]^{l} & \oplus_{i \in I\setminus I_0}\CG^v_{i}\ar[d]\ar[r]& 0 \\ 0\ar[r]
&	\Ker(l) \ar[r] & \oplus_{t(a)=v}P(\CH_{s(a)}) \ar[r]\ar[d]
& \CU \ar[r]\ar[d]&0 \\& & \oplus_{t(a)=v}C(\CH_{s(a)})\ar[d]\ar@{=}[r]& \oplus_{t(a)=v}C(\CH_{s(a)})\ar[d]& \\&&0&0&}	$$ Since the ending terms in the right-hand column are Gorenstein-projective, the same is true for $\CU$. Moreover, according to the right-most square, there is a short exact sequence

$$(*) \ \ \ 0 \rt \oplus_{t(a)=v}\CH_{{s(a)}} \st{{\left[\begin{smallmatrix}l\\ {\oplus_{t(a)=v}\gamma_{\CH_{s(a)}}}\end{smallmatrix}\right]}} \rt (\oplus_{i \in I\setminus I_0}\CG^v_{i}) \oplus ( \oplus_{t(a)=v}P(\CH_{s(a)}))\rt \CU\rt 0.$$ The above pull-back diagram gives the matrix presentation for the associated morphism $\CH^v:\oplus_{t(a)=v}\CH_{s(a)}\rt (\oplus_{i \in J\setminus J_0}\CG^v_{i}) \oplus (\oplus_{i \in J_0}\CG^v_{i})$ as $\CH^v=[\phi~~0]$, where $\phi={{\left[\begin{smallmatrix}l\\ {\oplus_{t(a)=v}\gamma_{\CH_{s(a)}}}\end{smallmatrix}\right]}}$. This implies that $\CH^v$ is injective. This in conjunction with the short exact sequence $(*)$, forces cokernel $\CU\oplus (\oplus_{i \in J_0}\CG^v_{i})$ to be Gorenstein projective. Thus the proof of the claim is completed. Continue this process until all vertices in $V_{\mu(\CQ)}$ are reached. This will give us
the desired representation $\CH$ in $\gpr\mbox{-}(\CQ, \La)$. So the proof is finished. \ \ \ \ \ \ \ \ \ \ \ \ \ \ \ \ \ \ \ \ \ \ \ \ \ \ \ \ \ \ \ \ \ \ \ \ \ \ \ \ \ \ \ \ \ \ \ \ \ \ \ \ \ \ \ \ \ \ \ \ \ \ \ \ \ \ \ \ \ \ \ \ \ \ \ \ \ $\square$\\

Similar to the classification given by \cite[Theorem C]{GKKP}, we obtain the following Gabriel-style classification of Gorenstein projective representations over G-semisimple algebras.
\begin{corollary}\label{dynkin}
Let $\CQ$ be an acyclic quiver, and let $\La$ be a {\rm G}-semisimple finite-dimensional algebra over an algebraically closed field $k$.
Let $m$ be the number of isomorphic classes of indecomposable non-projective Gorenstein projective modules. Then the following statements hold.
\begin{itemize}
\item [$(i)$] The path algebra $\La \CQ$ is $\cm$-finite if and only if $\CQ$ is Dynkin.
\item [$(ii)$] The number of isomorphism classes of indecomposable non-projective Gorenstein projective representation is $m|\Phi^+|,$ where $|\Phi^+|$ is the set of
positive roots of the (corresponding) Dynkin diagram.
\end{itemize}
\end{corollary}
\begin{proof}
Assume that $\{G_1, \cdots, G_m\}$ is the complete set of non-projective Gorenstein projective modules, up to isomorphism.
Set $G:=\oplus_{i=1}^mG_i$. Since ${\ad}\mbox{-} G\simeq\gpr\mbox{-}\La$, we have the isomorphism
$\md$-($\ugpr\mbox{-}\La) \simeq \md$-$A$, where $A=\underline{{\End}}_{\La}(\underline{G})$,
the endomorphism algebra of $G$ in the stable category $\ugpr\mbox{-}\La$. Moreover, based on our assumption for $\La$, as mentioned in Remark \ref{dd},
for every indecomposable module $\underline{X}$ and $\underline{Y}$ in $\ugpr \mbox{-} \La$, $\underline{\Hom}_{\La}(\underline{X}, \underline{Y})=0$, and
$\underline{\End}_{\La}(\underline{X})\simeq k$ as an isomorphism algebra. This observation yields that

$$ {\rm rep}(\CQ, \ugpr \mbox{-} \La) \simeq {\rm rep}(\CQ, \prod^m_{i=1}k ) \simeq \prod^m_{i=1} {\rm rep}(\CQ,k ).$$
Thanks to Theorem \ref{main}, there is a bijection between isomorphism classes of indecomposable non-projective Gorenstein projective representations in $\ugpr(\CQ, \La)$   and isomorphism classes of representations in ${\rm rep}(\CQ, \ugpr \mbox{-} \La)$. One the other hand, in view of the above equivalences,  the indecomposable non-projective objects in ${\rm rep}(\CQ, \ugpr \mbox{-} \La)$
 correspond to $m$-tuples $(\CX_1, \cdots, \CX_m)$, where $\CX_i$ is an
indecomposable representation in ${\rm rep}(\CQ, k)$. Furthermore,  by Gabriel’s theorem, there are only finitely many isomorphism
classes of indecomposable representations in ${\rm rep}(\CQ, k)$ if and only if $\CQ$ is Dynkin,  in which case they are in bijection with the positive roots of $\mathcal{Q}$. This completes the proof of the claim.
\end{proof}

\section{The stable  Auslander-Reiten quiver of path algebras over G-semisimple algebras}

In this section, we begin by establishing that the category of Gorenstein representations of the linear quiver $A_n$ over a G-semisimple algebra $\La$, denoted as $\gpr(A_n, \La)$, possesses almost split sequences (see Proposition \ref{Prop. general.functorial}). Subsequently, we present a complete classification of indecomposable Gorenstein projective representations within $\gpr(A_n, \La)$, assuming $\La$ is a G-semisimple algebra (see Proposition \ref{linearquiver}). Additionally, we determine almost split sequences in $\gpr(A_n, \La)$ that have specific ending terms. As an application, we apply this information to obtain insights into the cardinality of the components of the stable Auslander-Reiten quiver $\gpr(A_n, \La)$, as stated in our final theorem.

Let us begin with the following auxiliary lemma.

\begin{lemma}\label{lem. right.prep.Gp}
Let $\CK=(\CK_v, {\CK}_a)_{v \in V, a \in E}$ be a representation in ${\rm rep}(\CQ, \La)$ such that for every $v \in V$, $\CK_v \in (\gpr\mbox{-}\La)^{\perp}$. Then, for every every Gorenstein representation $\CG$, we have $\Ext^1_{\CQ}(\CG, \CK)=0$, i.e., $\CK \in (\gpr(\CQ, \La))^{\perp}.$
\end{lemma}
\begin{proof}
We proceed by induction on $V(\CQ)$. For the case $V(\CQ) = 1$, it is plainly true. Assume that the result holds for any quiver with $V(\CQ) < n$. Let $\CQ$ be a quiver with $V(\CQ) = n > 1$, and let $\CK$ be a representation such that for every $v \in V$, $\CK_v \in (\gpr\mbox{-}\La)^\perp$. Consider the subquiver $\CQ^{n-1}$ of $\CQ$. Assume that $\CG$ is a representation in $\gpr(\CQ, \La)$. By \cite{EH}, the restriction functor $e^{\CQ^{n-1}}: \text{rep}(\CQ, \La) \rightarrow \text{rep}(\CQ^{n-1}, \La)$ has left and right adjoints. According to the explicit formulations of the adjoints provided in \cite{AEHS}, we observe the following short exact sequence in $\gpr(\CQ, \La)$:
$$ 0 \rightarrow e^{\CQ^{n-1}}_{\la}e^{\CQ^{n-1}}(\CG) \rightarrow \CG \rightarrow \bigoplus_{v \in V_n} e^v_{\la}(\CG_v) \rightarrow 0. $$
Then, $\Ext^1_{\CQ}(e^{\CQ^{n-1}}_{\la}e^{\CQ^{n-1}}(\CG), \CK) \simeq \Ext^1_{\CQ^{n-1}}(e^{\CQ^{n-1}}(\CG), e^{\CQ^{n-1}}(\CK)) = 0$, due to our induction hypothesis. Moreover, $\Ext^1_{\CQ}(\bigoplus_{v \in V_n} e^v_{\la}(\CG_v), \CK) \simeq \bigoplus_{v \in V_n} \Ext^1_{\La}(\CG_v, \CK_v) = 0$. Now, by the long exact sequence associated with the above sequence, we get $\Ext^1_{\CQ}(\CG, \CK) = 0$, as required.
\end{proof}

\begin{lemma}\label{lemma, funct. A2}
Let $\gpr\mbox{-}\La$ be contravariantly finite in $\md\mbox{-}\La$. Then $\CS(\gpr\mbox{-}\La)$ is functorially finite in $\rm{H}(\La).$ In particular, $\CS(\gpr\mbox{-}\La)$ has almost split sequences.
\end{lemma}
\begin{proof}
Since $\CS(\gpr\mbox{-}\La)$ is a resolving subcategory of $\rm{H}(\La)$, by \cite[Corollary 0.3]{KS} it suffices to show that $\CS(\gpr\mbox{-}\La)$ is contravariantly finite in $\rm{H}(\La)$. It follows from \cite[Theorem 2.5]{RS2} (or \cite[Theorem 3.1]{LuZ}) that $\CS(\La)$ is contravariantly finite in $\rm{H}(\La)$. Thus, it is enough to show that any object in $\CS(\La)$ has a right $\CS(\gpr\mbox{-}\La)$-approximation. Assume that $(M_1\st{f}\rt M_2)$ is an arbitrary object in $\CS(\La)$. Take a minimal right $\gpr\mbox{-}\La$-approximation $G_1 \stackrel{g}\rightarrow \mathsf{cok}f\rightarrow 0$ of $\mathsf{cok}f$ in $\md\mbox{-}\La$, which exists by assumption. Set $K_1:=\mathsf{ker}g$. Since $g$ is minimal, Wakamutsu's Lemma yields that $K_1 \in \gpr\mbox{-}\La^{\perp}$, i.e., $\Ext^1(G, K_1)=0$ for any $G$ in $\gpr\mbox{-}\La$.
Consider the following pull-back diagram:
$${\xymatrix{& & 0\ar[d] & 0\ar[d] & &\\
& & M_1 \ar[d]^h
\ar@{=}[r] & M_1 \ar[d]^{f} \ar[r] & 0\\
0 \ar[r] & K_1\ar@{=}[d] \ar[r] & U
\ar[d]^l \ar[r]^{d} & M_2 \ar[d]^{\pi} \ar[r] & 0\\
0 \ar[r] & K_1 \ar[d] \ar[r]^{\mu_1} & G_1
\ar[d] \ar[r]^{g} & \mathsf{cok}f\ar[d] \ar[r] & 0\\
& 0 & 0 & 0 & }}$$
Let $G_2 \st{t}\rt U \rt 0$ be a minimal right $\gpr\mbox{-}\La$-approximation. Again by applying Wakamutsu's Lemma, $K_2:= \mathsf{ker}t$ belongs to $\gpr\mbox{-}\La^{\perp}$. Now consider the following pull-back diagram:
$$\xymatrix{& & 0 \ar[d] & 0 \ar[d]& &\\
0 \ar[r] & K_2 \ar@{=}[d] \ar[r] & K_3 \ar[d]
\ar[r] &K_1\ar[d] \ar[r] & 0\\
0 \ar[r] & K_2 \ar[r] & G_2
\ar[d]^{dt} \ar[r]^{t} & U\ar[d]^d \ar[r] & 0\\
& & M_2
\ar[d] \ar@{=}[r] & M_2 \ar[d] & \\
& & 0 & 0 & }	$$
Since both $K_1, K_2$ are in $\gpr\mbox{-}\La^{\perp}$, the same is true for $K_3$. Finally, applying the Snake lemma to the commutative diagram
$$\xymatrix{
0 \ar[r] &0 \ar[d] \ar[r] & G_2 \ar[d]^t
\ar@{=}[r] &G_2 \ar[d]^{lt}\ar[r] & 0&\\
0 \ar[r] & M_1 \ar[r]^{h} & U
\ar[r]^l & G_1\ar[r] &0& \\ } $$
gives the short exact sequence $0 \rt K_2 \rt G_3\rt M_1\rt 0, $
where $G_3:= \mathsf{ker}lt$, which evidently lies in $\gpr\mbox{-}\La$.
Putting together the maps obtained in the above, one may construct the following commutative diagram:
$$\xymatrix{& 0 \ar[d] & 0 \ar[d] & 0 \ar[d]& &\\
0 \ar[r] & K_2 \ar[d]^v \ar[r] & G_3 \ar[d]^w
\ar[r]^p & M_1 \ar[d]^{f} \ar[r] & 0\\
0 \ar[r] & K_3\ar[d]
\ar[r] & G_2
\ar[d]^{lt} \ar[r]^{dt} & M_2 \ar[d]^{\pi} \ar[r] & 0\\
0 \ar[r] & K_1 \ar[d] \ar[r] & G_1
\ar[d] \ar[r]^{g} & \mathsf{cok}f \ar[d] \ar[r] & 0\\
& 0 & 0 & 0 & }$$
{This, in particular, gives us the short exact sequence }
{$$\la:0\lrt(K_2\st{v}\rt K_3)\lrt (G_3\st{w}\rt G_2)\st{(p~~dt)}\lrt(M_1\st{f}\rt M_2)\lrt 0,$$} in $\rm{H}(\La)$	such that the middle term lies in $\CS(\gpr\mbox{-}\La)$ and $K_2, K_3$ are in $(\gpr\mbox{-}\La)^{\perp}$. Thus, in view of Lemma \ref{lem. right.prep.Gp}, the representation $(K_1\st{v}\rt K_2)$ is in $\CS(\gpr\mbox{-}\La)^{\perp}$. This implies that the epimorphism sitting in the short exact sequence $\la$ acts a right $\gpr\mbox{-}\La$-approximation of $(M_1\st{f} \rt M_2)$. So the proof is complete.
\end{proof}

In the sequel, we aim to extend the above result to the linear quiver $$A_n:v_1\st{a_1}\rightarrow \cdots \st{a_{n-1}}\rightarrow v_n,$$ where $n \geq 2.$ One should note that for the case $n=2$, $\gpr(A_2, \La)$ is the same as $\CS(\gpr\mbox{-}\La)$.

We recall that a representation $\CX=(\CX_v, \CX_a)_{v \in V, a \in E}$ of $\CQ$ over $\La$ is a {\em monic representation}, if for each $v \in V$, the $\La$-homorphism $\CX^v:\oplus_{t(a)=v}\CX_{s(a)} \rt \CX_v$ is injective (see \cite[Definition 2.2]{LuZ}). Denote by ${\mon}(\CQ, \La)$ the full subcategory of ${\rm rep}(\CQ, \La)$ consisting of monic representations of $\CQ$ over $\La$. It was shown in \cite[Theorem 3.1]{LuZ} that ${\mon}(\CQ, \La)$ is functorially finite in ${\rm rep}(\CQ, \La).$

\begin{proposition}\label{Prop. general.functorial}
Let $\gpr\mbox{-}\La$ be contravariantly finite in $\md\mbox{-}\La$. Then $\gpr(A_n, \La)$ is functorially finite in ${\rm rep}(\CQ, \La).$ In particular, $\gpr(A_n, \La)$ has almost split sequences.
\end{proposition}
\begin{proof}
Similar to the argument given at the beginning of the proof of Lemma \ref{lemma, funct. A2}, it is enough to show that any object in ${\mon}(A_n, \La)$ has a right $\gpr(A_n, \La)$-approximation. Take an arbitrary object $\CM=(\CM_{v_i}, \CM_{a_i})$ in ${\mon}(A_n, \La)$.
By using the induction on $i$, one may construct the short exact sequence in ${\rm rep}(A_n, \La)$, or in ${\mon}(A_n, \La)$,

$$(*) \ \ \ 0 \rt \CK \st{f}\rt \CG \st{g} \rt \CM \rt 0,$$
where $\CG=(\CG_{v_i}, \CG_{a_i})$ is a representation in $\gpr(A_n, \La)$, and $\CK=(\CK_{v_i}, \CK_{a_i})$ with $\CK_v \in (\gpr\mbox{-}\La)^{\perp}$ for each $1 \leq i \leq n$. By Lemma \ref{lem. right.prep.Gp}, the representation $\CK$ lies in $(\gpr(A_n, \La))^{\perp}$. This implies that the morphism $g: \CG \rt \CM$ is a right $\gpr(A_n, \La)$-approximation. For the case $i=2$, it holds, see the short exact sequence $\la$ in the proof of Lemma \ref{lemma, funct. A2}. Assume that we have construct the short exact sequence $(*)$ for every $v_j, j\in \{1, 2,\cdots, i\}$. More precisely, the representations $\CK, \CG$ and morphism $f, g $ have been defined for subquiver $A_i$ which satisfy the following conditions: for every $v_j, j \in \{1,2, \cdots, i\} $
\begin{itemize}

\item $0 \rt \CK_{v_j} \st{f_{v_j}}\rt \CG_{v_j} \st {g_{v_j}}\rt \CM_{v_j} \rt 0$ is a short exact sequence;
\item if $j<i$, then $\CG_{a_j}$ is a monomorphism, and $\mathsf{cok}(\CG_{a_j})$ is Gorenstein projective;
\item $\CG_{v_j}$ is Gorenstein projective, and $\CK_{v_j} \in (\gpr\mbox{-}\La)^{\perp}.$
\end{itemize}
We can write $g_{v_i}:\CG_{v_i}\rt \CM_{v_i}$ as $g_{v_i}=[q~~0]:\CG_{v_i}=X\oplus P\rt \CM_{v_i} $, where $q:X\rt \CM_v$ is a right minimal $\gpr\mbox{-}\La$-approximation. Since $P \in\gpr\mbox{-}\La \cap ({\gpr}\mbox{-}\La)^{\perp} $, one may deduce that $P$ is a projective module. Applying Lemma \ref{lemma, funct. A2} to the representation $(\CM_{v_i}\st{\CM_{a_{v_i}}}\rt \CM_{v_{i+1}})$ of the quiver $v_i\st{a_i}\rt v_{i+1}$, gives the short exact sequence

$$\delta: 0\lrt(K_1\st{v}\rt K_2)\lrt (G_1\st{w}\rt G_2)\st{(p_1~~p_2)}\lrt (\CM_{v_i}\st{\CM_{a_i}}\rt\CM_{v_{i+1}})\lrt 0,$$
where the middle term lies in $\CS(\gpr\mbox{-}\La)$ and $K_1, K_2$ are in $(\gpr\mbox{-}\La)^{\perp}$. Without loss of generality, we can assume that $X$ is a direct summand of $G_1$. This leads us to write $p_1$ as $p_1=[q~~0]:G_1=X\oplus Q\rt \CM_{v_i}$, for some projective module $Q$. Now, let us make a slight change to our record for the vertex $v_i$, by adding the split short exact sequence $0 \rightarrow Q \xrightarrow{1} Q \rightarrow 0 \rightarrow 0$ to the sequence $0 \rightarrow \CK_{v_i} \xrightarrow{f_{v_i}} \CG_{v_i} \xrightarrow{g_{v_i}} \CM_{v_i} \rightarrow 0$. Assume that $\CG_{v_{i+1}} = G_2 \oplus P$. In the diagram below, we visualize the relevant definitions in the neighborhood of the vertex $v_i$:

$$\xymatrix@C=0.7cm@R=0.7cm{0\ar[r]&
\CK_{v_{i-1}}\ar[d]_-{\left[\begin{smallmatrix} \CK_{a_{i-1}}\\ 0\end{smallmatrix}\right]}\ar[rr]^-{f_{v_{i-1}}}
&& \CG_{v_{i-1}}
\ar[d]_-{\left[\begin{smallmatrix} \CG_{a_{i-1}}\\ 0\end{smallmatrix}\right]}\ar[rr]^-{g_{v_{i-1}}}
&& \CM_{v_{i-1}}\ar[r]\ar[d]^-{\CM_{a_{i-1}}} & 0\\
0\ar[r]&
\CK_{v_i}\oplus Q\ar[d]_-{\CK_{a_i}}\ar[rr]^-{\left[\begin{smallmatrix} f_{v_i} & 0\\ 0 & 1\end{smallmatrix}\right]}
&& \CG_{v_i}\oplus Q=X\oplus P\oplus Q
\ar[d]^-{\left[\begin{smallmatrix}w|&0&w|\\0&1&
0\end{smallmatrix}\right]}
\ar[rr]^-{[g_{v_i}~~0~~0]}
&& \CM_{v_i}\ar[r]\ar[d]^-{\CM_{a_i}} & 0\\
0\ar[r]&
\CK_{v_{i+1}}\ar[rr]^-{f_{v_{i+1}}}
&& \CG_{v_{i+1}}=G_2\oplus P
\ar[rr]^-{[p_2~~0]}
&& \CM_{v_{i+1}}\ar[r] & 0,}$$
According to the diagram, we have defined $\CG_{a_i}$ as the morphism in the middle column, which is a monomorphism with a Gorenstein projective cokernel. We have also defined $\CK_{v_{i+1}}$ and $\CK_{a_{i}}$ as the module and the morphism, respectively, satisfying the commutativity of the leftmost square. So the proof is finished.
\end{proof}
\begin{remark}
We would like to emphasize that one may be able to extend the constructive method given in Proposition \ref{Prop. general.functorial} for finding a right $\gpr(A_n, \La)$-approximation for an object in ${\mon}(\CQ, \La)$ to finite acyclic quivers. But since the remainder of the paper deals with linear quivers and also to simplify the technicalities, we deal only with linear quivers.
\end{remark}

\subsection{The case: quiver $A_2:v_1 \rt v_2$}
We first specialize to G-semisimple algebras the results of \cite[Section 5]{HZ} concerning the components of the stable Auslander-Reiten quiver of $\CS(\gpr\mbox{-}\La)$ containing the boundary vertex $(0 \rt G)$. We then generalize to the case that the quiver is a linear quiver with $n$ vertices.

\begin{lemma}\label{AlmostSplittrivialmonomorphisms}
Let $\La$ be a ${\rm G}$-semisimple algebra. Let $G$ be a non-projective Gorenstein-projective module and $0 \rt \Omega_{\La}(G) \st{g} \rt P_G\st{f} \rt G \rt 0$ be an almost split sequence in $\gpr\mbox{-}\La$. Then
\begin{itemize}
\item[$(1)$] The almost split sequence in $\CS(\gpr\mbox{-}\La)$ ending at $(0 \rt G)$ is of the form
$$0\lrt(\Omega_{\La}(G)\st{1}\rt \Omega_{\La}(G))\st{(1~~g)}\lrt(\Omega_{\La}(G)\st{g}\rt P_G)\st{(0~~f)}\lrt(0\rt G)\lrt 0.$$

\item [$(2)$] The almost split sequence in $\CS(\gpr\mbox{-}\La)$ ending at $(G\st{1}\rt G)$ is of the form
$$0\lrt(\Omega_{\La}(G)\st{g}\rt P_G)\lrt(P_G\st{1}\rt P_G)\oplus(0\rt G)\lrt(G\st{1}\rt G)\lrt 0.$$

\item [$(3)$] The almost split sequence in $\CS(\gpr\mbox{-}\La)$ ending at $(\Omega_{\La}(G)\st{g}\rt P_G)$ is of the form	$$0\lrt(0\rt\Omega_{\La}(G))\lrt(0\rt P_G)\oplus(\Omega_{\La}(G)\st{1}\rt \Omega_{\La}(G))\lrt(\Omega_{\La}(G)\st{g}\rt P_G)\lrt 0.$$
\end{itemize}
\end{lemma}
\begin{proof}
$(1)$ The result follows immediately by applying \cite[Lemma 6.3(1)]{H} to the almost split sequence $0 \rt \Omega_{\La}(G) \st{g} \rt P_G\st{f} \rt G \rt 0$.\\ $(2)$ Consider the following push-out diagram in $\md\mbox{-}\La$
$$\xymatrix{
0 \ar[r] &\Omega_{\La}(G) \ar[d]^g \ar[r]^g & P_G \ar[d]^{\left[\begin{smallmatrix} 1\\ f \end{smallmatrix}\right]}
\ar[r]^f &G \ar@{=}[d]\ar[r] & 0&\\
0 \ar[r] & P_G \ar[r]^{{\left[\begin{smallmatrix}1\\0\end{smallmatrix}\right]}} & P_G\oplus G
\ar[r]^{[1~~0]} & G\ar[r] &0.& \\ } $$
In view of \cite[Lemma 6.3(2)]{H}, this diagram induces the almost split sequence in $\CS(\gpr\mbox{-}\La)$

$$0\lrt(\Omega_{\La}(G)\st{g}\rt P_G)\lrt (P_G\st{{{\left[\begin{smallmatrix}1\\f\end{smallmatrix}\right]}}}\rt P_G\oplus G)\lrt (G\st{1}\rt G)\lrt 0,$$which is the desired one, because
the middle object in this short exact sequence is isomorphic to $(P_G\st{1}\rt P_G)\oplus (0 \rt G)$. This can be understood by referring to the classification of indecomposable objects in $\CS(\gpr\mbox{-}\La)$, as provided in Proposition \ref{pro11}. \\
$(3)$ Consider the pull-back diagram
$$\xymatrix{ & \Omega_{\La}(G) \ar[d]^{\left[\begin{smallmatrix} -1 \\ g\end{smallmatrix}\right]}
\ar@{=}[r] & \Omega_{\La}(G) \ar[d]^g \\
\Omega_{\La}(G)\ar@{=}[d] \ar[r]^{\left[\begin{smallmatrix} 1 \\ 0\end{smallmatrix}\right]} & \Omega_{\La}(G)\oplus P_G
\ar[d]^{\left[\begin{smallmatrix} g & 1\end{smallmatrix}\right]} \ar[r]^{\left[\begin{smallmatrix} 0 & 1\end{smallmatrix}\right]} & P_G \ar[d]^f \\
\Omega_{\La}(G) \ar[r]^{g} & P_G
\ar[r]^{f} & G}	$$
By the use of \cite[Lemma 6.3(3)]{H}, the top and the middle rows give us the almost split sequence in $\CS(\gpr\mbox{-}\La)$

$$0\lrt(0\rt\Omega_{\La}(G))\lrt (\Omega_{\La}(G)\st{1}\rt\Omega_{\La}(G)\oplus P_G)\rt (\Omega_{\La}(G)\st{g}\rt P_G)\lrt 0.$$One should note that according to Proposition \ref{pro11}, the middle term in this sequence is isomorphic to $(\Omega_{\La}(G)\st{1}\rt \Omega_{\La}(G))\oplus (0 \rt P_G)$, and so, it is the desired almost split sequence. Thus the proof is complete.
\end{proof}

\begin{remark}
The middle term of the almost split sequence appeared in the first assertion of Lemma \ref{AlmostSplittrivialmonomorphisms} is always indecomposable. But the middle terms of almost split sequences that appeared in the second and third assertions are a direct sum of an indecomposable and a projective object. Since the projective direct summand may not be indecomposable, the middle terms have two or more direct summands. However, the stable Auslander-Reiten quivers, obtained by removing the projective-injective vertices, for each case of the almost split sequence in the above lemma, display a mesh with only one vertex in the middle.
\end{remark}

\begin{setup}\label{Setup1} Assume that $G$ is a non-projective $\La$-module and $n$ is the least positive integer such that $\Omega^n_{\La}(G)= G$. Take the exact sequence obtained from the minimal projective resolution of $G$
$$0 \rt \Omega^n_{\La}(G)\st{p_n}\rt P^{n-1}\st{p_{n-1}}\rt \cdots\rt P^1\st{p_1}\rt P^0\st{p_0}\rt G\rt 0.$$
For any $ 0 \leqslant i \leqslant n-1$, denote $\epsilon_i: 0\rt \Omega_{\La}^{i+1}(G)\st{j_i}\rt P^i\st{r_i}\rt \Omega_{\La}^i(G)\rt 0$.
\end{setup}
Let $G$ be a non-projective module in $\gpr\mbox{-}\La$. Denote by $\Gamma^s_{G}$ the component of the stable Auslander-Reiten quiver of $\CS(\gpr\mbox{-}\La)$ containing the vertex $(0 \rt G)$.

\begin{theorem}\label{Theorem 4.10}
Keeping the above notations, let $\La$ be a ${\rm G}$-semisimple algebra and $G$ a non-projective object in $\gpr\mbox{-}\La$ with $\Omega^n_{\La}(G)= G$, where $n=l(G)$, i.e., $n$ is the minimum positive integer for which $\Omega^n_{\Lambda}(G)\simeq G$. The following assertions hold.

\begin{itemize}
\item [$(1)$] The vertices of the component $\Gamma^s_G$ are corresponded to the following indecomposable objects
$$(0\rt\Omega_{\La}^{i+1}(G)), \\~~ (\Omega_{\La}^{i+1}(G)\st{1}\rt\Omega_{\La}^{i+1}(G))~~ \\~~ {\rm and } ~~(\Omega_{\La}^{i+1}(G)\st{j_i}\rt P^i),$$
where $i$ runs through $0\leqslant i \leqslant n-1$. In particular, $\Gamma^s_{G}$ has $3n$ vertices.

\item [$(2)$] The stable Auslander-Reiten quiver $\Gamma^s_{\CS(\gpr\mbox{-}\La)}$ is a disjoint union of all components $\Gamma^s_G$, where $G$ belongs to a fixed set of representatives of all equivalence classes in $\CC(\La)$. In particular, there is a bijection between the components of $\Gamma^s_{\CS(\gpr\mbox{-}\La)}$ and the elements of $\CC(\La)$.
\end{itemize}
\end{theorem}
\begin{proof}
Applying Lemma \ref{AlmostSplittrivialmonomorphisms} (1) to the almost split sequence $\epsilon_0$ and Lemma \ref{AlmostSplittrivialmonomorphisms} (2) and (3) to the almost split sequence $\epsilon_1 $ (if $n=1$, applying for $\epsilon_0, \epsilon_0$), one may reach the following part of $\Gamma_{\CS(\gpr\mbox{-}\La)}$
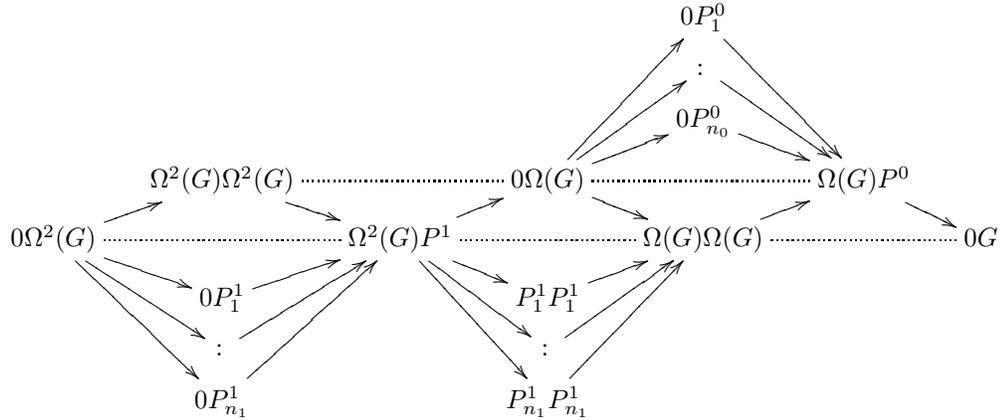
\[
\xymatrix @R=0.2cm @C=0.5cm {
&&&&&& 0P^0_1\ar[rddd]&&&\\&&& & && \colon\ar[rdd]&&&\\&&&&&&0P^0_{n_0}\ar[dr]& &&	\\& &&\Omega^2(G)\Omega^2(G)\ar[dr]\ar@{.}[rr]&	&0\Omega(G)\ar[dr]\ar[ru]\ar[ruu]\ar[ruuu]\ar@{.}[rr]&&\Omega(G)P^0\ar[dr]&&&\\& &0\Omega^2(G)\ar[dr]\ar[ru]\ar[rddd]\ar[rdd]\ar@{.}[rr]&&
\Omega^2(G)P^1\ar[dr]\ar@{.}[rr]\ar[ru]\ar[rdd]\ar[rddd]&&\Omega(G)\Omega(G)\ar[ru]\ar@{.}[rr]&&0G&&\\
&&&0P^1_1\ar[ur]&	&P^1_1P^1_1\ar[ru]&&&&&\\&&&\colon\ar[ruu]&& \colon\ar[ruu]&&&&\\&&& 0P^1_{n_1}\ar[ruuu]&&P^1_{n_1}P^1_{n_1}\ar[ruuu]&&&&}
\]where $P^0=\oplus_{i=1}^{n_0}P^0_{i}$ and $P^1=\oplus_{i=1}^{n_1}P^1_{i}$. To save space, we have omitted the maps and the subscript ``$\La$".
Repeating the same as construction for the pair of the short exact sequences $(\epsilon_1, \epsilon_2)$ until to the pair $(\epsilon_{n-2}, \epsilon_{n-1})$, we will obtain $n-1$ full subquivers of $\Gamma_{\CS(\rm{Gprj}\mbox{-}\La)}$ as the above such the one corresponding to $(\epsilon_{n-2}, \epsilon_{n-1})$ has the object $(0 \rt G)$ in the leftmost side. Hence the construction will stop at $(n-1)$-th step. By glowing the obtained full subquivers we obtain the full subquiver $\tilde{\Gamma}_G$ of $\Gamma_{\CS(\gpr\mbox{-}\La)}$ containing the $\tau_{\CG}$-orbit of $(0\rt G)$. Be deleting the projective-injective vertices of the full subquiver $\tilde{\Gamma}_G$, we get the component $\Gamma^s_G,$ as presented in the below

{\footnotesize{\[
\xymatrix @R=0.3cm @C=0.4cm {
&\Omega^n(G)\Omega^n(G)\ar[rd]\ar@{.}[r]&&&\cdots&	&0\Omega(G)\ar[dr]\ar@{.}[rr]&&\Omega(G)P^0\ar[dr]&\\0\Omega^n(G)\ar[ru]\ar@{.}[rr]&&\Omega^n(G)P^{n-1}&&\cdots&
\Omega^2(G)P^1\ar@{.}[rr]\ar[ru]&&\Omega(G)\Omega(G)\ar[ru]\ar@{.}[rr]&&0G}\]}}We have to identify the vertex corresponding to $(0 \rt \Omega^n_{\La}(G))$ with to $(0\rt G)$ as by our assumption $\Omega^n_{\La}(G)=G$. By the construction, we observe that the vertices of $\Gamma^s_G$ are determined uniquely with indecomposable modules in the equivalence class $[G]$. This establishes a bijection from $\CC(\La)$ to the set of the components $\Gamma^s_{\CS(\gpr\mbox{-}\La)}$, giving the second assertion. Simultaneously, our proof gives the classification of the component of the stable Auslander-Reiten quiver of $\CS(\gpr\mbox{-}\La)$, asserted in the first assertion. So the proof is finished.
\end{proof}

According to the proof of Theorem \ref{Theorem 4.10}, we observe that the module $(0 \rightarrow G)$ has $\tau_{\CS}$-periodic length $3(l(G)-1)$.
This observation will be generalized in Lemma \ref{AlmostSplitNmon}.
\begin{remark}
Denote $Z_n$ a basic $n$-cycle with the vertex set $\{1, 2,\cdots, n \}$
and the arrow set $\{\alpha_1, \alpha_2, \cdots, \alpha_n\}$, where $s(\alpha_i)=i$ and $t(\alpha_i)=i+1$ for each $1\leqslant i \leqslant n$. Here, we identify $n+1$ with 1. Let $I$ denote the arrow ideal of the path algebra $kZ_n$ over a field $k$. We recall that $I$ is generated by all arrows in $Z_n.$ Assume that $\La$ is a finite dimensional algebra over an algebraic closed field $k$ of characteristic different from two. Let $\Gamma^s_{G_1}, \cdots, \Gamma^s_{G_n}$ denote all the components, and denote by $d_i$ the number of vertices in $\Gamma^s_{G_i}$ for each $1\leqslant i \leqslant n$. Our assumption on $\La$ guarantees that the components of $\Gamma^s_{\CS(\rm{Gprj}\mbox{-}\La)}$ must be standard. Hence, in view of Theorem \ref{Theorem 4.10},
there is an equivalence
$$\text{ind}\mbox{-}\underline{\CS(\gpr\mbox{-}\La)}\simeq kZ_{d_1}/I^2_1 \times \cdots \times kZ_{d_n}/I^2_n,$$
where $kZ_{d_i}/I^2_i$ is the quotient category of the path category $kZ_n$ modulo the ideal $I^2_i$, the radical square ideal, indeed, generated by all paths of length $\geqslant 2$, and $3 | d_i$. Moreover, if $\La$ is Gorenstein, $T_2(\La)$ is also Gorenstein \cite[Corollary 4.3]{AHKe1}. Using a result of Buchweitz \cite{Bu}, we obtain that the equivalence $\mathbb{D}_{\rm{sg}}(T_2(\La))\simeq \underline{\CS(\gpr\mbox{-}\La)}$. Therefore, we have $$\text{ind}\mbox{-}\mathbb{D}_{\text{sg}}(T_2(\La))\simeq kZ_{d_1}/I^2_1 \times \cdots \times kZ_{d_n}/I^2_n.$$ This equivalence yields a corresponding equivalence:

\begin{equation}\label{equivalencesing} \mathbb{D}_{\text{sg}}(T_2(\La))\simeq \text{prj}\mbox{-} kZ_{d_1}/I^2_1 \times \cdots \times \text{prj}\mbox{-}kZ_{d_n}/I^2_n.
\end{equation}

Based on Happel's work, the left side $\underline{\CS(\gpr\mbox{-}\La)}$ of the equivalence \eqref{equivalencesing} forms a triangulated category. Consequently, we can infer from this equivalence that the category of projective objects in any quotient path category $kZ_{3n}/I^2$ (on the right side) can be endowed with an enriched triangulated structure. It would be of interest to investigate this induced triangulated structure on the projective category of each quotient category $\mathsf{prj}\mbox{-} kZ_{3n}/I^2$. By considering all the categories $\mathsf{prj}\mbox{-} kZ_{d_i}/I^2_i$ equipped with the induced triangulated structure, the additive equivalence in \eqref{equivalencesing} becomes a triangle equivalence.
\end{remark}

The result below, which generalizes \cite[Corollary 3.1]{Ka}, applies
Theorem \ref{Theorem 4.10} to interpret $\CC(\La)$ in terms of the number of components of $\Gamma^s_{\CS(\gpr\mbox{-}\La)}$.

We recall that two Artin algebras $\La$ and $\La'$ are said to be derived equivalent if there exists an equivalence of triangulated categories $\mathbb{D}^{\rm{b}}(\md\mbox{-}\La)\simeq \mathbb{D}^{\rm{b}}(\md\mbox{-}\La')$. It is proven that if two Artin algebras $A$ and $A'$ are derived equivalent, then there exists a triangulated equivalence $\ugpr\mbox{-}A\simeq \ugpr\mbox{-}A'$. This result is shown in \cite[Theorem 4.1.2]{AHV} and also in \cite{HP} using the notion of the stable functor.

\begin{corollary}
Let $\La$ and $\La'$ be {\rm G}-semisimple algebras (not necessarily finite-dimensional algebras). If $\La$ and $\La'$ are derived equivalent, then there is a bijection
$$f: \CC(\La)\rt \CC(\La')$$
such that $l(f([G]))=l([G])$ for all $[G] \in \CC(\La).$
\end{corollary}
\begin{proof}
Since $\mathbb{D}^{\rm{b}}(\md\mbox{-}\La)\simeq \mathbb{D}^{\rm{b}}(\md\mbox{-}\La')$, specializing \cite[Theorem 8.5]{A} to the case that the quiver is $\bullet\rt \bullet$,
gives us the derived equivalence $\mathbb{D}^{\rm{b}}(\md\mbox{-} T_2(\La))\simeq \mathbb{D}^{\rm{b}}(\md\mbox{-}T_2(\La'))$. This combined with \cite[Theorem 4.1.2]{AHV},
enables us to infer that $\ugpr\mbox{-}T_2(\La)\simeq \ugpr\mbox{-}T_2(\La')$. Now the triangle equivalences $\ugpr\mbox{-}T_2(\La)\simeq \underline{\CS(\gpr\mbox{-}\La)} $ and $\ugpr\mbox{-}T_2(\La')\simeq \underline{\CS(\gpr\mbox{-}\La')} $, yield the equivalence
$\underline{\CS(\gpr\mbox{-}\La)}\simeq \underline{\CS(\gpr\mbox{-}\La')}$. This implies that the stable Auslander-Reiten quivers of $\Gamma^s_{\CS(\gpr\mbox{-}\La)}$ and $\Gamma^s_{\CS(\gpr\mbox{-}\La')}$ are isomorphic. Now Theorem \ref{Theorem 4.10}(2) completes the proof.
\end{proof}

Assume that $\La$ and $\La'$ are quadratic monomial algebras that are derived equivalent. According to the above corollary and Remark 2.7, we can conclude that the perfect components of the relation quiver $\mathcal{R}_\Lambda$ are in bijection with those of $\mathcal{R}_{\Lambda'}$.

\subsection{The case: quiver $A_n$}

Let $A_n=v_1 \rt \cdots \rt v_n$ be the linear quiver with $n \geq 1$ vertices. The path algebra $\La A_n$ is given by the lower triangular
matrix algebra of $\La$
$$T_n(\La)=
\left[ \begin{array}{cccc}

\La & 0 & \cdots & 0 \\
\La & \La & \cdots & 0 \\
\vdots & \vdots & \ddots & \vdots \\
\La & \La & \cdots & \La \\
\end{array} \right]
$$
Moreover, a representation $X$ of $A_n$ over $\La$ is a datum as described in below:
$$X=(X_1\st{f_1}\rt X_2\st{f_2}\rt \cdots \st{f_{n-1}}\rt X_n)$$
where $X_i$ is a $\La$-module and $f_i:X_i\rt X_{i+1}$ is a $\La$-homomorphism. A morphism from representation $X$ to representation $Y$ is a datum $(\alpha_i:X_i\rt X_{i+1})_{1\leqslant i \leqslant n-1}$ such that for each $1\leqslant i \leqslant n-1$
\[\xymatrix{X_i \ar[r]^{f_i} \ar[d]_{\alpha_2} & X_{i+1} \ar[d]^{\alpha_{i+1}} \\
Y_i \ar[r]^{g_i} & Y_{i+1} }\]
commutes. Denote by $\CS_n(\gpr\mbox{-}\La):=\gpr(A_n, \La).$ In particular, $\CS_2(\gpr\mbox{-}\La)=\CS(\gpr\mbox{-}\La).$

\begin{proposition}\label{linearquiver}
Let $\La$ be a {\rm G}-semisimple algebra. Then any indecomposable representation in $\CS_n(\gpr\mbox{-}\La)$ is isomorphic to a representation of the following form: either
$$ {[i, j, G]}:=(0 \rt \cdots\rt 0 \rt \Omega_{\La}(G) \st{1}\rt \Omega_{\La}(G) \cdots \rt \Omega_{\La}(G) \st{1}\rt \Omega_{\La}(G) \st{l} \rt P \st{1}\rt \cdots \st{1}\rt P) $$
where $G$ is an indecomposable Gorenstein projective module, $1\leqslant i \leqslant j \leqslant n,$ the first $\Omega_{\La}(G)$ is settled in the $i$-th vertex and the last one in the $j$-th vertex, the map $l$ attached to the arrow $v_j\rt v_{j+1}$ is the inclusion $\Omega_{\La}(G)\hookrightarrow P$, or
$$ {[0, 0, P]}:=(P \st{1}\rt P \rt \cdots \rt P \st{1}\rt P ) $$
where $P$ is an indecomposable projective module, and all maps are the identity on $P.$
\end{proposition}
\begin{proof}
Since the case $n=1$ is clear, we assume $n\geq 2.$ By the local characterization of Gorenstein projective representations given in Theorem \ref{Gorenproj charecte}, a representation $X=( X_1 \st{f_1} \rt \cdots \rt X_{n-1} \st{f_{n-1}} \lrt X_n)$ in $\rm{rep}(A_n, \La)$ is Gorenstein projective if and only if the following conditions hold:
\begin{itemize}
\item [$(1)$] For any $1 \leq i \leq n$, $X_i$ is a Gorenstein projective module.
\item [$(2)$] For any $1 \leq i \leq n-1,$ $\mathsf{cok}f_i$ is a Gorenstein projective module and $f_i$ is a monomorphism.
\end{itemize}
Assume that $X=(X_1\st{f_1}\rt X_2\rt \cdots \rt X_{n-1}\st{f_{n-1}}\rt X_n)$ is an indecomposable Gorenstein projective representation in $\text{rep}(\CQ, \La)$ which is not isomorphic to a representation of the form $[0, 0, P]$, for some indecomposable projective module $P$. We claim that it is isomorphic to a representation given in the statement. At the end of the proof, we will show that any representation of that form is indecomposable.

Assume that $m$ is the least integer such that $X_m\neq 0.$ We prove the claim by the inverse induction on $m$. If $m=n$, then the case is clear, as $X=[n, n, \Omega^{-1}_{\La}(X_n)]$.. Assume that $m<n.$ Denote by $X'$ the sub-representation of $X$ such that $X'_m=0$ and $X'_d=X_d$ for all $m < d \leqslant n$. Lemma \ref{Lemma1} and the indecomposability of $X$, would imply that the monomorphism $f_m:X_m\rt X_{m+1}$ is isomorphic to one of the cases: $(1): (0\rt G), (2):( G\st{1}\rt G), (3): (\Omega_{\La}(G)\hookrightarrow P)$. The first case is impossible as $X_m$ would be 0, which is a contradiction. If the second case holds, then one may easily observe that $\End_{A_n}(X)\simeq \End_{A_n}(X')$. So $X'$ is indecomposable and then, the characterization given in Theorem \ref{Gorenproj charecte}, forces $X'$ to be Gorenstein projective. Hence $X'$ is an indecomposable Gorenstein projective representation with $X'_m=0$. Our inductive assumption implies that $X' \simeq {[m+1, j, G]}$ for some $m+1 \leqslant j\leqslant n$, and so, $X\simeq {[m, j, G]}$. If the third case holds, then we show that $X \simeq {[m, m, G]}$. In the case $m+1=n$, there is nothing to prove, as $X\simeq [n-1, n-1, G]$. So assume that $m+1<n$. According to Lemma \ref{Lemma1}, there are three cases for the monomorphism $f_{m+1}:P\rt X_{m+2}$, similar to what we discuss for the monomorphism $f_m$. The first and third cases are impossible. Only the second case holds. This allows us to identify $f_{m+1}$ with the identity map of $P$. Continuing the same argument for the remaining morphisms implies that they also may be identified with the identity map on $P$. Thus, we obtain the desired form.

Now we prove a representation ${[i, j, G]}$ as in the statement is indecomposable. Assume that $\alpha=(\alpha_i)_{1\leqslant i \leqslant n-1}$ is an endomorphism of ${[i, j, G]}$. Since $\alpha$ satisfies the commutativity conditions, $\alpha_i=\alpha_d=\alpha_j$ for $i \leqslant d \leqslant j$ and $\alpha_{j+1}=\alpha_s=\alpha_n$ for $j+1\leqslant s \leqslant n.$ If $\alpha_i=\alpha_d, (i\leqslant d \leqslant j)$ are automorphisms, then by applying this fact that $l$ is a minimal left $\prj\La$-approximation, one deduces that $\alpha_{j+1}=\alpha_s, (j+1\leqslant s \leqslant n)$ are also automorphisms. Hence, $\alpha$ is an automorphism. If $\alpha_i=\alpha_d$ for all $i\leqslant d \leqslant j$ are not automorphisms, then they are nilpotent as $\Omega_{\La}(G)$ is indecomposable. The commutative diagram associated with the arrow $v_{j}\rt v_{j+1}$, yields the following commutative diagram:
$$\xymatrix{
0 \ar[r] &\Omega_{\La}(G) \ar[d]^{\alpha_j} \ar[r]^l & P \ar[d]^{\alpha_{j+1}}
\ar[r] &G \ar[d]^{\beta}\ar[r] & 0&\\
0 \ar[r] & \Omega_{\La}(G) \ar[r]^{l} & P
\ar[r] & G\ar[r] &0.& \\ } $$
But the induced map $\beta$ is not an automorphism. Otherwise, as $P\rt \mathsf{cok}l=G$ is a projective cover, $\alpha_{j+1}$ is also automorphism, and so, the above diagram yields that $\alpha_j$ is an automorphism, which is a contradiction. Now since $G$ is indecomposable, $\beta$ will be nilpotent. This implies that $\alpha_{j+1}=\alpha_s$ for $j+1\leqslant s \leqslant n$ are nilpotent, and then, the same is true for $\alpha$. Thus the proof is complete.
\end{proof}

According to the terminology introduced in Proposition \ref{linearquiver}, we can determine all indecomposable projective representations in ${\rm rep}(A_n, \La)$ as follows: $[i,j, P]$, where $0 \leqslant i \leqslant j \leqslant n$ and $P$ is an indecomposable projective module. It should be noted that there are alternative ways to present the projective representations, apart from the form $[0, 0, G]$. For example, when $n=5$, the representations $[1,3, P]$, $[2, 3, P]$, and $[3, 3, P]$ all express the same projective representation $(0\rt 0\rt 0 \rt P \st{1}\rt P)$. However, indecomposable non-projective Gorenstein projective representations are uniquely determined up to isomorphism by the terminology provided in the proposition.

As an immediate consequence of Proposition \ref{linearquiver}, we include the result below.
\begin{corollary}
Let $G_1, \cdots, G_m$ be a set of pairwise non-isomorphic indecomposable non-projective representations, and let $P_1, \cdots, P_s$ be a set of pairwise non-isomorphic indecomposable projective representations. Then, any indecomposable Gorenstein projective representation is of one of the following forms:
\begin{itemize}
\item [$(1)$] $[i, j, G_k]$, where $1\leqslant i \leqslant j \leqslant n$, and $1 \leqslant k \leqslant m;$
\item [$(2)$] $[i,j, P_k]$, where $0\leqslant i \leqslant j \leqslant n$, and $1 \leqslant k \leqslant s.$
\end{itemize}
These forms constitute a complete list of pairwise non-isomorphic indecomposable representations in ${\rm rep}(A_n, \La)$. In particular, there are $ns+\frac{mn(n+1)}{2}$ indecomposable objects up to isomorphism in ${\rm rep}(A_n, \La).$
\end{corollary}

The above corollary states that the path algebra $\La A_n$, or equivalently $T_n(\La)$, is $\cm$-finite when $\La$ is G-semisimple. By use of \cite[Corollary 7.6]{ABHV}, which states the stable category of $\gpr(Q, \La)$ is equivalent to that of $\gpr(Q', \La)$, where $Q'$ is a quiver with the same underlying graph with acyclic quiver $A$ and $\La$ Gorenstein algebra, we can observe that $\La B$ is $\cm$-finite as well, where $B$ is a quiver with the same underlying graph with $A_n$. This leads to produce more $\cm$-finite algebras.

Based on the previous result, we think that utilizing the local description of Gorenstein projective representations provided in \cite[Theorem 5.1]{LuZ} for a finite acyclic quiver, as well as the corresponding result in \cite{LZh2} for acyclic quivers with monomial relations, has the potential to generate additional $\cm$-finite algebras. These can be obtained by considering (quotient algebras of) path algebras of G-semisimple algebras.

\begin{lemma}\label{AlmostSplitNmon}
Let $\La$ be a G-semisimple algebra and let $G$ be an indecomposable non-projective Gorenstein-projective module. Then
\begin{itemize}
\item[$(1)$] The almost split sequence in $\CS_n(\gpr\mbox{-}\La)$ ending at $[n, n, \Omega^{-1}_{\La}G]$ is of the form
$$\xymatrix@1{ 0\ar[r] & [1, n, G]
\ar[r]
& [1,n-1, G]\ar[r]&
[n, n, \Omega^{-1}_{\La}(G)]\ar[r]& 0. } \ \ $$

\item [$(2)$] The almost split sequence in $\CS_n(\gpr\mbox{-}\La)$ ending at $[1, n, G]$ is of the form
$$\xymatrix@1{ 0\ar[r] &[1,1, \Omega_{\La}(G)]
\ar[r]
&[0,0, P^1]\oplus [2, n, G]\ar[r]&
[1,n, G]\ar[r]& 0. } \ \ $$

\item [$(3)$] Let $1 \leqslant i < n$. The almost split sequence in $\CS_n(\gpr\mbox{-}\La)$ ending at $[i,i, \Omega_{\La}(G)]$ is of the form
$$\xymatrix@1{ 0\ar[r] & [i+1, i+1, \Omega_{\La}(G)]
\ar[r]
& [1, i+1, P^1]\oplus [i, i+1, \Omega_{\La}(G)] \ar[r]&
[i, i, \Omega_{\La}(G)]\ar[r]& 0. } \ \ $$

\end{itemize}

\end{lemma}
\begin{proof} Throughout the proof, let $I_n=\{1, 2, \cdots, n\}$, and consider the setup described in Setup \ref{Setup1}. We assume that representations in $\CS_n(\gpr\mbox{-}\La)$ are direct sums of the indecomposable representations listed in Proposition \ref{CM-finite}.\\ $(1)$ Since $\CS_n(\gpr\mbox{-}\La)$ has almost split sequences, there exists an almost split sequence
$$\la: 0 \rt {\rm X}\rt {\rm Y}\rt [n, n, \Omega^{-1}_{\La}(G)]\rt 0,$$
where ${\rm X}=(X_j)_{j \in I_n}$ and ${\rm Y}=(Y_j)_{j \in I_n}$.
Let $\CQ_1$ be the subquiver $v_{n-1} \st{a_{n-1}}\rt v_n$ of $A_n.$ Apply the exact functor $e^{\CQ_1}$ to the sequence $\la$ and get the short exact sequence in $\gpr(\CQ_1, \La)= \CS(\gpr\mbox{-}\La)$ ending at $(0 \rt G)$, which we denote as $\la_1$. By using the fact that $\la$ is an almost split sequence in $\CS_n(\gpr\mbox{-}\La)$ and considering the structure of indecomposable representations in $\CS_n(\gpr\mbox{-}\La)$ and $\CS(\gpr\mbox{-}\La)$, one may deduce that $\la_1$ is an almost split sequence in $\CS(\gpr\mbox{-}\La)$. According to Lemma \ref{AlmostSplittrivialmonomorphisms}, $$0\lrt ( \Omega_{\La}(G)\st{1}\rt \Omega_{\La}(G))\lrt ( \Omega_{\La}(G)\st{r_0}\rt P^0)\lrt (0\rt G)\lrt 0$$ is an almost split sequence in $\CS(\gpr\mbox{-}\La)$. Hence, $X_{n-1}=X_n=\Omega_{\La}(G)$, and $Y_n=P^0, Y_{n-1}=\Omega_{\La}(G)$. Therefore, in view of the list of indecomposable representations given in Proposition \ref{AlmostSplittrivialmonomorphisms}, the possible choices for ${\rm X}$ and ${\rm Y}$ are $[1, n, G]$ and $[1, n-1, G]$, respectively.\\
$(2)$ Another use of the fact that $\CS_n(\gpr\mbox{-}\La)$ has almost split sequences, one gets an almost split sequence
$$\beta : 0 \rt {\rm X}\rt {\rm Y}\st{\phi}\rt [1, n, G]\rt 0,$$
where ${\rm X}=(X_i)_{i \in I_n}$ and ${\rm Y}=(Y_i)_{i \in I_n}$. Let $\CQ_2$ be the subquiver $v_1\st{a_1}\rt v_2$ of $A_n$. By applying the exact functor $e^{\CQ_2}$ to $\beta$, we obtian the short exact sequence $\beta_1$ in $\gpr(\CQ_2, \La)= \CS(\gpr\mbox{-}\La)$ ending at $(\Omega_{\La}(G)\st{1} \rt \Omega_{\La}(G))$. The short exact sequence $\beta_1$ does not split. Otherwise, this implies that ${\rm Y}$ contains $[1, i, G]$, for some $1 \leqslant i \leqslant n$, as a direct summand. If $i<n, $ according to the arrow $v_i\st{a_i}\rt v_{i+1}$, the epimorphism lying in $\beta$ gives us the commutative diagram
\[\xymatrix{\Omega_{\La}(G) \ar[r]^{j_0} \ar@{=}[d] & P^0\ar[d]^{\phi_{i+1}} \\
\Omega_{\La}(G) \ar@{=}[r] & \Omega_{\La}(G) }\]
which implies that $j_0$ is a section, a contradiction. Hence, if $i=n$, then $\beta$ splits, again leading to a contradiction. We then use this fact that $\beta$ is an almost split sequence in $\CS_n(\gpr\mbox{-}\La)$ and deduce that $\beta_1$ is an almost split sequence in $\CS(\gpr\mbox{-}\La).$ Thus, Lemma \ref{AlmostSplittrivialmonomorphisms} implies that $\beta_1$ has to be the following sequence
$$0\lrt(\Omega_{\La}^2(G)\st{j}\rt P^1)\lrt(P^1\st{1}\rt P^1)\oplus(0\rt\Omega_{\La}(G))\lrt(\Omega_{\La}(G)\st{1}\rt\Omega_{\La}(G))\lrt 0.$$Hence, $X_1=\Omega^2_{\La}(G)$ and $X_2=P^1$. This yields that the only choice for the indecomposable representation ${\rm X}$ is $[1, 1, \Omega_{\La}(G)]$. Finally, we use the fact that on every vertex we have a short exact sequence in $\gpr\mbox{-}\Lambda$ to deduce that the middle term $Y$ is the same as the middle term given in the sequence of $(2)$ in the statement. This completes the proof of $
(2)$.\\$(3)$ Since the case $n=2$ follows from Lemma \ref{AlmostSplittrivialmonomorphisms}, we assume that $n>2.$ Let $1 \leqslant i < n$. Consider the almost split sequence in $\CS_n(\gpr\mbox{-}\La)$ given by:
$$\delta_i: 0 \rightarrow {\rm X} \rightarrow {\rm Y} \rightarrow [i, i, \Omega_{\La}(G)] \rightarrow 0,$$
where ${\rm X}=(X_j)_{j \in I_n}$ and ${\rm Y}=(Y_j)_{j \in I_n}$. Now, let $\CQ_i=v_i\rt v_{i+1}$ be a subquiver of $A_n$. Apply the exact functor $e^{\CQ_i}$ to the sequence $\delta_i$ and get the short exact sequence $\delta_i'$ in $\gpr(\CQ_i, \La)= \CS(\gpr\mbox{-}\La)$. An easy argument reveals that $\delta'_i$ is an almost split sequence in $\CS(\gpr\mbox{-}\La)$. Hence, Lemma \ref{AlmostSplittrivialmonomorphisms} provides us the almost split sequence in $\CS(\gpr\mbox{-}\La)$
$$0\lrt(0\rt\Omega^2_{\La}(G))\lrt(0\rt P^1)\oplus(\Omega^2_{\La}(G)\st{1}\rt\Omega^2_{\La}(G))\lrt(\Omega^2_{\La}(G)\st{j_1}\rt P^1)\lrt0.$$This gives rise to the equalities $X_0=\cdots=X_i=0, X_{i+1}=\Omega^2_{\La}(G)$, and $Y_i=\Omega^2_{\La}(G), Y_{i+1}=P^1\oplus \Omega^2_{\La}(G)$. Thanks to the list of indecomposable representations appeared in Proposition \ref{AlmostSplittrivialmonomorphisms}, we infer that ${\rm X}=[i+1, m, \Omega_{\La}(G) ]$ and ${\rm Y}=[i, m', \Omega_{\La}(G)]\oplus [i, i, P^1]$, where $i+1 \leqslant m \leqslant n$ and $i \leqslant m' \leqslant n$. Note that $[i, i, P^1]=[1, i+1, P^1$. We first show that $m'\neq n.$ Assume to the contrary $m'=n.$ Let $\CQ':v_i\st{a_i}\rt v_{i+1}\cdots \rt v_n$ be a subquiver of $A_n$. Since all the terms of the sequence $\delta_i$ are zero for the vertices $v_j, 1 \leqslant j \leqslant i-1$, $e^{\CQ'}(\delta_i)$ can be considered as an almost split sequence in $\CQ'.$ If we apply $(1)$ of the lemma to the indecomposable representation $[1, n-i+1, \Omega_{\La}(G)]$ in $\gpr(\CQ', \La)$, we need to use the assumption that $n> 2$. This application gives us the middle term of an almost split sequence starting from $[1, n-i+1, \Omega_{\La}(G)]$, which is an indecomposable representation $[1, n-i, \Omega_{\La}(G)]$ in $\gpr(\CQ', \La)$. However, this would imply that $\Omega^2_{\La}(G)=P^1$, which leads to a contradiction. Since restricting the sequence $\delta_i$ on vertex $v_n$ induces a short exact sequence in $\gpr\mbox{-}\La$, we have ${\rm X}_n=P^1, {\rm Y}_n=P^1\oplus P^1$. Now one may observe that the almost sequence $e^{\delta_i}(\CQ')$ induces an almost split sequence on the subquiver $v_i\st{a_i}\rt v_{i+1}\rt \cdots \rt v_{n-1}$. By repeating the similar argument on the vertex $v_{n-1}$ and continuing this process until the vertex $v_{i+2}$, we conclude that for each $ i+2 \leqslant j \leqslant n$, ${\rm X}_j=P^1$ and ${\rm Y}_j=P^1\oplus P^1$. Therefore, ${\rm X}$ and ${\rm Y}$ are in the desired forms, and so, the proof is complete.
\end{proof}

\begin{remark}\label{rmk. middlet-term}
Upon examining the almost split sequences mentioned in Lemma \ref{AlmostSplitNmon}, it becomes apparent that their middle terms include a non-projective summand. This observation leads us to conclude that the vertex of the $\tau_{\CS_n}$ orbit of the form $[1, 1, \Omega_{\La}(G)]$ remains in the same component within $\Ga^s_{\CS_n}$.
\end{remark}

\color{black}

In \cite{HZ}, an interesting three-periodicity phenomenon was investigated. Specifically, it was shown that, under certain mild conditions, the number of finite components in the stable Auslander-Reiten quiver of $\CS(\gpr\mbox{-}\La)=\CS_2(\gpr\mbox{-}\La)$ containing a boundary vertex is divisible by 3 (see \cite[Theorem 6.2]{HZ}). In what follows, we will explore a generalization of this type of periodicity for the finite components of the stable Auslander-Reiten quiver of $\CS_n(\gpr\mbox{-}\La)$, where $\La$ is a G-semisimple algebra.

Denote by $\tau_{\CS_n}$ the Auslander-Reiten translation in $\CS_n(\gpr\mbox{-}\La)$. We need the following definition within the proof of the next theorem.
\begin{definition}
Assume that $\Delta=(\Delta_0, \Delta_1)$ is a quiver with the sets $\Delta_0$ and $\Delta_1$ of vertices and arrows, respectively. The repetition of $\Delta$, denoted by
$\mathbb{Z}\Delta$, is a quiver which is defined as follows:
\begin{itemize}
\item [$\mbox{-}$] $(\mathbb{Z}\Delta)_0=\mathbb{Z}\times\Delta_0$
\item [$\mbox{-}$] $(\mathbb{Z}\Delta)_1=\mathbb{Z}\times \Delta_1 \cup \sigma(\mathbb{Z}\times \Delta_1)$ with arrows $(n,\alpha):(n,x)\rightarrow (n,y)$ and
$\sigma(n,\alpha):(n-1,y)\rightarrow (n,x)$ for each arrow $\alpha:x\rightarrow y$ in $\Delta_1$ and $n \in \mathbb{Z}$.
\end{itemize}
\end{definition}

The quiver $\Z\Delta$ with the translation $\tau(n,x)=(n-1,x)$ is clearly a stable translation quiver that does not depend (up to isomorphism) on the
orientation of $\Delta$, see \cite{Rie}.
Let $\Delta=\mathbb{A}_n$, and define an automorphism $S$ of $\Z\Delta$ by sending $(p, q)$ to $(p+q, n+1-q)$.\\

By abuse of notation, we identify the vertex of the Auslander-Reiten quiver and indecomposable modules. $\Ga_{\CS_n}$ (resp. $\Ga^s_{\CS_n}$ ) denote the (resp. stable) AR-quiver of $\CS_n({\gpr}\mbox{-}\La)$, obtained by removing projective-injective vertices in $\Ga_{\CS_n}.$

\begin{lemma}\label{component-boundary}
Let $\La$ be a  G-semisimple algebra and let $\Ga$ be a component of the stable Auslander-Reiten quiver of $\CS_n(\gpr\mbox{-}\La)$ which has at least one arrow. Then, it contains a vertex of the form $[n, n, \Omega^{-1}_{\La}(G)]$ for some indecomposable non-projective Gorenstein projective module $G$.
\end{lemma}
\begin{proof}
Without loss of generality we may assume that $\La$ is an indecomposable, i.e., $\prj\La$ is a connected category. This implies that $ T_n(\La)$ is an indecomposable algebra. Due to the equivalence $\md$-$T_n(\La) \simeq {\rm rep}(A_n, \La)$, one infers that $\prj T_n(\La)$ is a connected category. Since $\CS_n(\gpr\mbox{-}\La)$ contains $\prj T_n(\La)$, the subcategory of projective representations over $A_n$ is connected. As $\CS_n(\gpr\mbox{-}\La)$ is of finite representation type which contains the connected subcategory of projective representations, \cite[Lemma 5.1]{Li} yields that $\Ga_{\CS_n}$ is connected. It means that there is a walk between any two vertices.

Assume that $\Ga$ is an arbitrary component of $\Ga^s_{\CS_n}$. Take a vertex $x$ of $\Ga.$ The connectedness of $\Ga_{\CS_n}$ gives a walk $y=x_0\longleftrightarrow x_1 \longleftrightarrow \cdots x_{t-1}\longleftrightarrow x_t=x $ such that $y$ is a projective-injective vertex and for $1\leq d\leq t$, $x_d$ is not a projective vertex. By $x_d\longleftrightarrow x_{d+1}$ we mean that there is either an arrow $x_d\rt x_{d+1}$ or an arrow $x_{d+1}\rt x_d$ in $\Ga_{\CS_n}$. According to the characterization of indecomposable projective representations in ${\rm rep}(A_n, \La)$, there exists an indecomposable projective module $P$ that we can determine $y=[i, j, P]$ for some $0 \leqslant i \leqslant j \leqslant n$. Since $\La$ is $G$-semisimple, any almost split sequence in $\gpr\mbox{-}\La$ has a projective middle term. Now as $\gpr\mbox{-}\La$ is a connected category and is of finite representation type, $P$ must appear as a direct summand of the term of minimal projective resolution of an indecomposable non-projective Gorenstein projective module $G$. One should note that since $\Gamma$ has at least one arrow, $\gpr\mbox{-}\La$ can not be equal to $\prj\La.$ Otherwise, all Gorenstein projective representations over $A_n$ turns to be projective, implying that $\Gamma=\varnothing$, which is a contradiction. According to Lemma \ref{AlmostSplitNmon}, $y$ is the middle term of an almost split sequence which has a non-projective summand. Hence $x_1$ must be in the $\tau_{\CS_n}$-orbit of the vertex $[n, n, \Omega^{-1}_{\La}(G)]$. In view of Remark \ref{rmk. middlet-term}, the vertex $[n, n, \Omega^{-1}_{\La}(G)]$ lies in the component $\Ga.$ So the proof is finished.
\color{black}
\end{proof}

Note that the stable Auslander-Reiten quiver of $\CS_n(\gpr\mbox{-}\La)$ is nothing else than the Auslander-Reiten quiver of the triangulated category $\underline {\CS_{n}(\gpr\mbox{-}\La)}$.

\begin{theorem}
Let $\La$ be a finite dimensional {\rm G}-semisimple algebra over an algebraically closed field. Let $\Ga$ be a component of the stable Auslander-Reiten quiver of $\CS_n(\gpr\mbox{-}\La)$. Then

\begin{itemize}
\item [$(1)$] if $n$ is even, then the cardinality of $\Ga$ is divisible by $n+1$;
\item [$(2)$] if $n$ is odd, then the cardinality of $\Ga$ is divisible by $\frac{n+1} {2}.$
\end{itemize}
\end{theorem}
\begin{proof}

Since $\Ga$ is a component of the stable Auslander-Reiten quiver of the triangulated category $\underline {\CS_{n}(\gpr\mbox{-}\La)}$, we obtain by \cite{Am} or \cite{XZ} that $\Ga=\Z\Delta/G$, where $G\simeq \Z$, is a weakly
admissible automorphism group and $\Delta$ is a Dynkin quiver. One can find a complete list of all possible generators of the group $G$ in \cite[Theorem 2.2.1]{Am} or \cite[Section 4.3]{Rie}. In view of Proposition \ref{CM-finite}, there is no almost split sequence in $\CS_n(\gpr\mbox{-}\La)$ with a middle term having more than two indecomposable direct summands. This implies that $\Delta $ must be of type $\mathbb{A}_m$ for some positive integer $m.$ Without loss generality, one may take an orientation of the Dynkin quiver $\Delta$ as below

\begin{displaymath}
\xymatrix{\mathbb{A}_m:\quad 1\ar[r] & 2\ar[r] & \cdots \ar[r]& m-1\ar[r] & m.}
\end{displaymath}
According to \cite[Theorem 2.2.1]{Am}, all possible generators of $G$ are listed follows:
\begin{itemize}
\item
if $m$ is odd, then possible generators are
$\tau^r$ and $\phi\tau^r$ with $r\geq 1$, where $\phi=\tau^\frac{m+1}{2} {S}$
is an automorphism of $\Z\mathbb{A}_m$ of order 2.
\item
if $m$ is even, then possible generators are $\rho^r$, where $r\geq 1$ and
$\rho=\tau^{\frac{m}{2}} {S}$. Since $\rho^2=\tau^{-1}$, $\tau^r$ is a possible generator.
\end{itemize}
Lemma \ref{component-boundary} implies that there exists a vertex of the form $[n, n, \Omega^{-1}_{\La}(G)]$, for some indecomposable Gorenstein projective module $G.$ Then Lemma \ref{component-boundary} yields that $n+1$ applications of the Auslander-Reiten translation $\tau_{\CS_n}$ to the object $[n, n, \Omega^{-1}_{\La}(G)]$ is of $\tau_{\CS_n}$-period $n+1$. Let us fix the following notation throughout the proof.
\begin{itemize}
\item Let $\alpha$ denote the possible generator of $G$.
\item For every $i \in \Z$ and a vertex $a$ in $\Delta$, we let $\widetilde{(i, a)}$ denote the associated $\alpha$-orbit of the vertex $(i, a)$ in $\Z\Delta$, that is, the set of all $\alpha^j(i, a)$ where $j$ runs through $\Z$ .
\item Due to the identification of $\Gamma$ and $\mathbb{Z}\Delta/G$, we correspond the vertex ${\bf X}=[n, n, \Omega^{-1}_{\Lambda}(G)]$ in $\Gamma$ to the orbit of a vertex $(0, x)$ in $\mathbb{Z}\Delta$, where $x$ is a vertex in $\Delta$. In other words, ${\bf X}=\widetilde{(0, x)}$. Therefore, $\widetilde{(0, x)}$ is $\tau$-period of $n+1$.

\item Let $U(v)$ denote the $\tau$-orbit of a vertex $v$ in $\Ga=\Z\Delta/G$.
\end{itemize}

We divide the proof of $(1)$ and $(2)$ into two steps.
\begin{itemize}
\item [$(i)$] If $m$ is even, then $\alpha=\rho^r$ for some $r\geq 1$.
As $\Ga=\Z\mathbb{A}_{m}/<\rho^r>$, we observe that $U(\widetilde{(0, j)})=U(\widetilde{(0, m+1-j)})$ for every $1\leq j\leq m$, and
$$\mid\Ga \mid=\Sigma^{\frac{m}{2}}_{j=1}\mid U(\widetilde{(0, j)}) \mid=\frac{m}{2}r.$$
Keeping in mind that $\rho^2=\tau^{-1}$, we get the vertices $(0, x)$ and $\rho^{2r}(0, x)=\tau^{-r}(0, x)$ are clearly in the same $\rho^r$-orbit. Thus $\tau^r_{\CS}\widetilde{(0, x)}=\widetilde{(0, x)}$. Hence, $n+1$ divides $\mid \Ga \mid$, and consequently $(1)$ and $(2)$ hold.
\item [$(ii)$] If $m$ is odd, then there are two possibilities either $\alpha=\tau^r$ or $\alpha=\phi\tau^r$ for some $r \geq 1$. For the first case, as $\Ga=\Z\mathbb{A}_{m}/<\tau^r>$, one may easily see that $\mid \Ga \mid=rm$. Hence, $n+1$ divides $\mid \Ga \mid$. This implies $(1)$ and $(2)$. For the latter case, which $\alpha=\phi \tau^r$, we first observe that $U(\widetilde{(0, j)})=U(\widetilde{(0, m+1-j)})$ for every $1\leq j\leq n$, and since $\Ga=\Z\mathbb{A}_{m}/<\phi \tau^r>$, one has the equality
$$\mid\Ga \mid=\Sigma^{\frac{m+1}{2}}_{j=1}\mid U(\widetilde{(0, j)}) \mid=\frac{m+1}{2}r.$$
\end{itemize}
Since $\phi^2$ is an automorphism of order 2, and the commutativity of $\tau$ and $S$, we infer that $(\phi\tau^r)^2(0, x)=\tau^{2r}(0, x).$ Hence $(0, x)$ and $\tau^{2r}(0, x)$ lie in the same $\phi\tau^2$-orbit. This leads that $\tau^{2r}_{\CS} {\bf X}=\tau^{2r}_{\CS}\widetilde{(0, x)}=\widetilde{(0, x)}={\bf X}$. Hence $n+1$ divides $2r$. If $n$ is even, then $n+1$ divides $r$. Hence, $(1)$ holds. If $n$ is even, then $\frac{n+1}{2}$ divides $r$. Hence, we obtain $\mid \Ga \mid$ is divisible by $\frac{n+1}{2}$. This completes the proof.
\end{proof}

Let us conclude this section with a simple example indicating that for an odd integer $n$, it is not correct that the cardinality of $\Gamma$ is divided by $n+1$. To demonstrate this, take $n=3$ and $\Lambda=k[x]/(x^2)$. In this case, the only component of the stable AR-quiver of $\mathcal{S}_3(\md\mbox{-}k[x]/(x^2))$ has a cardinality of 6, see \cite[Example 2.5]{XZZ}.

\section*{Acknowledgments}
The authors are grateful to Sondre Kvamme for pointing out
an error in an earlier draft.

\end{document}